\documentclass[a4paper,reqno]{amsart}
\usepackage{geometry}
\usepackage{amsfonts}
\usepackage{amsmath,amssymb,amsthm,amsxtra}
\usepackage{float}
\usepackage[colorlinks, linkcolor=blue!72,anchorcolor=orange,
    citecolor=blue,urlcolor=Emerald, bookmarksopen,bookmarksdepth=2]{hyperref}
\usepackage[usenames,dvipsnames]{xcolor}
\usepackage{enumitem}
\usepackage{geometry,array}
\usepackage{graphicx}
\usepackage{subfigure}
\usepackage{bookmark}
\usepackage{tikz}
\usepackage{url}

\usetikzlibrary{matrix,positioning,decorations.markings,arrows,decorations.pathmorphing,	
    backgrounds,fit,positioning,shapes.symbols,chains,shadings,fadings,calc}
\tikzset{->-/.style={decoration={  markings,  mark=at position #1 with
    {\arrow{>}}},postaction={decorate}}}
\tikzset{-<-/.style={decoration={  markings,  mark=at position #1 with
    {\arrow{<}}},postaction={decorate}}}

\usepackage[all]{xy}

\def\del{\delta}

\usepackage{color}
\def\red{\color{red}}
\numberwithin{equation}{section}

\theoremstyle{plain}

\newtheorem{thm}{Theorem}[section]
\newtheorem{cor}[thm]{Corollary}
\newtheorem{lem}[thm]{Lemma}
\newtheorem{prop}[thm]{Proposition}

\theoremstyle{definition}
\newtheorem{defn}[thm]{Definition}
\newtheorem{exm}[thm]{Example}
\newtheorem{rem}[thm]{Remark}

\newdir{ >}{{}*!/-5pt/\dir{>}}

\newcommand{\Hom}{\operatorname{Hom}\nolimits}

\newcommand{\Ext}{\operatorname{Ext}\nolimits}

\newcommand{\proj}{\operatorname{proj}\nolimits}

\newcommand{\pd}{\operatorname{pd}\nolimits}
\newcommand{\id}{\operatorname{id}\nolimits}
\newcommand{\Id}{\operatorname{Id}\nolimits}

\newcommand{\gl}{\operatorname{gl.dim}\nolimits}

\newcommand{\sil}{\operatorname{silt}\nolimits}
\newcommand{\til}{\operatorname{tilt}\nolimits}
\renewcommand{\sup}{\operatorname{sup}\nolimits}

\renewcommand{\mod}{\mathsf{mod}\hspace{.01in}}
\newcommand{\Cone}{\operatorname{Cone}\nolimits}
\newcommand{\CoCone}{\operatorname{CoCone}\nolimits}

\newcommand{\B}{\mathcal B}

\newcommand{\oB}{\overline{\B}}
\newcommand{\U}{\mathcal U}
\newcommand{\V}{\mathcal V}

\newcommand{\W}{\mathcal W}

\newcommand{\s}{\mathcal S}
\newcommand{\T}{\mathcal T}

\newcommand{\D}{\mathcal D}

\newcommand{\R}{\mathcal R}

\newcommand{\Z}{\mathcal Z}
\newcommand{\C}{\mathcal C}

\newcommand{\EE}{\mathbb E}

\newcommand{\svecv}[2]{\left(\begin{smallmatrix}
      #1 \\
      #2
    \end{smallmatrix}\right)}

\newcommand{\svech}[2]{\left(\begin{smallmatrix}
      #1 & #2
\end{smallmatrix}\right)}

\renewcommand{\emph}{\textit}
\renewcommand{\phi}{\varphi}

\newcommand{\add}{\mathsf{add}\hspace{.01in}}
\newcommand{\thick}{\mathsf{thick}\hspace{.01in}}

\begin{document}

\title{Silting reduction in extriangulated categories}

\thanks{Yu Liu is supported by the National Natural Science Foundation of China (Grant No. 11901479). Panyue Zhou is supported by the National Natural Science Foundation of China (Grant No. 11901190) and by the Scientific Research Fund of Hunan Provincial Education Department (Grant No. 19B239). Yu Zhou is supported by National Natural Science Foundation of China (Grant No. 11801297) and by Tsinghua University Initiative Scientific Research Program (2019Z07L01006). Yu Zhou and Bin Zhu are supported by National Natural Science Foundation of China (Grant No. 12031007).}

\author{Yu Liu, Panyue Zhou, Yu Zhou and Bin Zhu}
\address{School of Mathematics, Southwest Jiaotong University, 610031 Chengdu, Sichuan, China}
\email{liuyu86@swjtu.edu.cn}
\address{College of Mathematics, Hunan Institute of Science and Technology, 414006 Yueyang, Hunan, China}
\email{panyuezhou@163.com}
\address{Yau Mathematical Science Center, Tsinghua University, Beijing 10084, China}
\email{yuzhoumath@gmail.com}
\address{Department of Mathematical Sciences, Tsinghua University, Beijing 10084, China}
\email{zhu-b@mail.tsinghua.edu.cn}

\begin{abstract}
 Presilting and silting subcategories in extriangulated categories were introduced by Adachi and Tsukamoto recently. In this paper, we prove that the Gabriel-Zisman localization $\mathcal B/(\thick\mathcal W)$ of an extriangulated category $\B$ with respect to a presilting subcategory $\mathcal W$ satisfying certain condition can be realized as a subfactor category of $\mathcal B$. This generalizes the result by Iyama-Yang for silting reduction on triangulated categories. Then we discuss the relation between silting subcategories and tilting subcategories in extriangulated categories, this gives us a kind of important examples of our results. In particular, for a finite dimensional Gorenstein algebra, we get the relative version of the description of the singularity category due to Happel and Chen-Zhang by this reduction.
\end{abstract}
\keywords{extriangulated categories; silting subcategories; silting reduction; Gorenstein categories}
\subjclass[2020]{18G80; 18E10}
\maketitle

\section{Introduction}
The concept of \emph{silting}  appeared for the first time in a paper of Keller-Vossieck \cite{KV}. Nowadays, silting theory and silting reduction play important roles in the research of triangulated categories, especially derived categories. In \cite{KY}, it is shown that silting objects have correspondence with many important structures such as $t$-structures, co-$t$-structures, simple-minded collections. Silting reduction is shown to have a close relation with Calabi-Yau reduction \cite{IY,IY1} and is widely used in the representation theory. An overview of recent developments in silting theory  can be found in \cite{A}. Motivated by these fruitful results, we are wondering if we can study such important theory under a more general setting.

Extriangulated category was introduced by Nakaoka and Palu \cite{NP} by extracting the relevant properties of ${\rm \Ext}^1$ on exact categories and triangulated categories.  In particular, triangulated categories and exact categories are extriangulated categories. An important source of examples of extriangulated categories  which are neither triangulated nor exact is extension closed subcategories of triangulated categories. To find more examples of extriangulated categories, see \cite{NP,ZZ,HZZ,ZhZ,NP1}.

Adachi and Tsukamoto \cite{AT} introduced the notion of presilting and silting subcategories (see Definition \ref{presil} and Definition \ref{def:sil}) in extriangulated categories as a generalization of presilting and silting subcategories in triangulated categories. Note that in module categories, this notion is different from the concept of silting module in the sense of Angeleri H\"{u}gel, Marks and Vit\'{o}ria \cite{AMV}.

Throughout this paper, $k$ denotes a field and $\B$ denotes a Krull-Schmidt, Hom-finite, $k$-linear extriangulated category with enough projectives and enough injectives and satisfies \cite[Condition 5.8]{NP} (see condition {\bf (WIC)} in page 4).  Cotorsion pairs in extriangulated categories were introduced by Nakaoka and Palu \cite{NP} as a generalization of those in triangulated categories and exact categories.  We recall its definition.
\begin{defn}
Let $\U,\V$ be a pair of full subcategories of $\B$ which are closed under direct sums and direct summands. $(\U,\V)$ are called a cotorsion pairs if it satisfies the following conditions.
\begin{itemize}
\item[(a)] $\EE(\U,\V)=0$.
\item[(b)] For any object $B\in \B$, there exist two $\EE$-triangles
$$V_B\to U_B\to B\dashrightarrow, \quad B\to V^B \to U^B\dashrightarrow$$
where $U_B,U^B\in \U$ and $V_B,V^B\in \V$.
\end{itemize}
\end{defn}

Let $\W$ be a presilting subcategory of $\B$ which satisfies the following condition (the readers can find the definitions of the notions in Definition \ref{Cone} and Page 5):
\begin{description}
\item[(CP)] $(\W^{\vee},\W^{\bot})$ and $({^{\bot}}\W,\W^{\wedge})$ are cotorsion pairs in $\B$.
\end{description}
We show that $\Z:={^{\bot}}\W \cap\W^{\bot}$ is a Frobenius extriangulated category where $\W$ is the subcategory of projective-injective objects (see Lemma~\ref{lem:Fro}). Hence by \cite[Theorem 3.13]{ZZ}, there is an induced triangulated structure on $\Z/[\W]$. One the other hand, let $\thick \W$ be the thick subcategory of $\B$ generated by $\W$, and $\R$ be the class of morphisms in $\B$ defined as following:
$$\R=\{f:A\to B\mid\text{ there exists an } \EE\mbox{-triangle} \text{ } A\xrightarrow{f} B\to M\dashrightarrow \text{ with } M\in \thick \W\}.$$
Assume that $\B$ is skeletally small. We show that there is a natural triangulated structure on the Gabriel-Zisman localization $\mathcal B/(\thick\mathcal W)$ of $\B$ with respect to $\R$. The main result of this paper is the following.

\begin{thm}[{Theorems~\ref{main5} and \ref{main6}, and Corollary~\ref{cor:corr}}]\label{main1.7}
Let $\W$ be a presilting subcategory of $\B$ which satisfies condition {\bf (CP)}. There is a triangle equivalence
$$\B/(\thick \W)\simeq \Z/[\W].$$
Moreover, this equivalence induces a one-to-one correspondence between the silting subcategories in $\B/(\thick\W)$ and the silting subcategories in $\B$ which contain $\W$.
\end{thm}

This silting reduction generalizes \cite[Theorem 1.1]{IY}.

In extriangulated categories, tilting subcategories were introduced by Zhu and Zhuang \cite{ZhZ} as a generalization of $n$-tilting modules \cite{M} (and dually there is the notion of cotilting subcategories). Note that in triangulated categories, this notion is different from the concept of tilting subcategory since  the only projective objects in a triangulated category are zeros.  In this article, we give a slightly different definition of tilting subcategory (compared with \cite{ZhZ}, see Definition~\ref{deftil}) and  study the relation between silting subcategories and tilting subcategories.

\begin{thm}[{Theorem~\ref{main4}}]\label{thm:main1.1}
The following statements hold.
\begin{itemize}
\item[\rm (1)] If any object $B\in \B$ has finite projective dimension, then any tilting subcategory is silting.
\item[\rm (2)] If the projective dimensions of objects in $\B$ have a common upper bound, then any silting subcategory is also tilting.
\item[\rm (3)] If there is an object $B\in B$ having infinite projective dimension, then any silting (resp. tilting) subcategory is not tilting (resp. silting).
\end{itemize}
\end{thm}

A class of presilting subcategories satisfying condition (CP) are contained in the following result, which generalizes a result of Happel-Unger \cite[Lemma 1.3]{HU} and a result of Xie-Zan-Zhang \cite[Theorem 2.6]{XZZ}.

\begin{prop}[Theorem~\ref{main4} and Lemma~\ref{lem:eq}]\label{prop1}
If $\B$ is an abelian category, then the following statements are equivalent:
\begin{itemize}
\item[\rm (1)] $\B$ is Gorenstein;
\item[\rm (2)] any tilting subcategory is cotilting;
\item[\rm (3)] any cotilting subcategory is tilting;
\item[\rm (4)] $\B$ has a tilting-cotilting subcategory.
\end{itemize}
Under one of these equivalent conditions, any tilting subcategory of $\B$ satisfies condition $\operatorname{(CP)}$.
\end{prop}

Applying Theorem~\ref{main1.7} and Proposition~\ref{prop1} to a finite dimensional Gorenstein algebra $A$, we can get that for any finitely generated  $n$-tilting $A$-module $T$, there are triangle equivalences
$$\mod A/(\thick T)\simeq  (T^{\bot}\cap {^{\bot}}T)/[T] \simeq \mathrm{D}^b(A)/\mathrm{K}^b(\add T),$$
where the second equivalence appeared in \cite[Theorem 2.5]{CZ}.

This article is organized as follows. In Section 2, we recall the definition of  silting subcategory in extriangulated category and some basic properties of it. In Section 3, we show that the localization $\mathcal B/(\thick\mathcal W)$ with respect to a presilting subcategory $\mathcal W$ satisfying condition (CP) can be realized as a subfactor category of $\mathcal B$, which is a generalization of the well-known silting reduction.  In Section 4, we show that the localization $\mathcal B/(\thick\mathcal W)$ has a natural  triangulated structure. In Section 5, we  study the relation between silting subcategories and tilting subcategories in extriangulated categories, the results provide a kind of important examples of our silting reduction.

\section{Preliminaries}

\subsection{Extriangulated categories}

We briefly recall the definition and basic properties of extriangulated categories, see \cite{NP} for more details.

Denote by ${\rm Ab}$ the category of abelian groups. Let $(\B, \mathbb{E}, \mathfrak{s})$ be an extriangulated category in the sense of \cite{NP}, where $\B$ is an additive category equipped with an additive bifunctor
$$\EE: \B^{\rm op}\times \B\rightarrow {\rm Ab},$$
and $\mathfrak{s}$ is an \emph{additive realization} of $\mathbb{E}$ in the sense of \cite[Definitions~2.9 and 2.10]{NP}, satisfying certain axioms listed in \cite[Definition~2.12]{NP}. For convenience, we usually say $\B$ is an extriangulated category.

For any objects $A, C$ of $\B$, an element $\delta\in\mathbb{E}(C,A)$ is called an \emph{$\EE$-extension}. Since $\EE$ is an additive bifunctor, for any $\EE$-extensions $\delta\in\EE(C,A)$ and any morphisms $a:A\to A',c:C\to C'$, we have $\EE$-extensions
$$a_\ast\delta:=\EE(C,a)(\delta)\in\EE(C,A')\text{ and }c^\ast\delta:=\EE(c,A)(\delta)\in\EE(C',A).$$
For any $\mathbb{E}$-extension $\delta\in\mathbb{E}(C,A)$, there is an associated class $\mathfrak{s}(\delta)$ of sequences $A\xrightarrow{~x~}B\xrightarrow{~y~}C$ of morphisms in $\B$, where $x$ is called an {\it inflation} and $y$ is called a {\it deflation}. In this case, the pair $(A\xrightarrow{~x~}B\xrightarrow{~y~}C,\delta)$ is called an {\it $\EE$-triangle}, and is written as
$$A\overset{x}{\longrightarrow}B\overset{y}{\longrightarrow}C\overset{\delta}{\dashrightarrow}.$$

Let $A\overset{x}{\longrightarrow}B\overset{y}{\longrightarrow}C\overset{\delta}{\dashrightarrow}$ and $A^{\prime}\overset{x^{\prime}}{\longrightarrow}B^{\prime}\overset{y^{\prime}}{\longrightarrow}C^{\prime}\overset{\delta^{\prime}}{\dashrightarrow}$ be any pair of $\EE$-triangles. Let $a:A\to A'$ and $c:C\to C'$ be any pair of morphisms such that $a_\ast\delta=c^\ast\delta'$. There exists $b:B\to B'$ such that the following diagram commutes.
$$\xymatrix{
A \ar[r]^x \ar[d]^a & B\ar[r]^y \ar[d]^{b} & C\ar@{-->}[r]^{\del}\ar[d]^c&\\
A'\ar[r]^{x'} & B' \ar[r]^{y'} & C'\ar@{-->}[r]^{\del'} &}$$
This triplet $(a,b,c)$ is called a {\it morphism of $\EE$-triangles}. We also call this diagram a \emph{commutative diagram of $\EE$-triangles}.

The following property of extriangulated categories is used frequently in the proofs of this paper.

\begin{prop}[{\cite[Proposition 1.20]{LN}}]\label{prop:NPLN}
Let $A\overset{x}{\longrightarrow}B\overset{y}{\longrightarrow}C\overset{\delta}{\dashrightarrow}$
be any $\EE$-triangle and $a\colon A\to A'$ be any morphism. Then there exists
a morphism $b$ which gives a morphism of $\EE$-triangles
$$\xymatrix{
A \ar[r]^x \ar[d]^a & B\ar[r]^y \ar[d]^{b} & C\ar@{-->}[r]^{\del}\ar@{=}[d]&\\
A'\ar[r]^{u} & B' \ar[r]^{v} & C\ar@{-->}[r]^{a_{\ast}\del} &}$$
such that $A\xrightarrow{\svecv{-a}{x}} A'\oplus B\xrightarrow{\svech{u}{b}} B'\stackrel{v^*\delta}\dashrightarrow$
is an $\EE$-triangle.
\end{prop}

For convenience, we usually omit the $\EE$-extension in the $\EE$-triangle that realizes it.

\subsection{Syzygy and cosyzygy}

An object $P\in\B$ is called {\it projective} if
for any $\EE$-triangle
$$A\xrightarrow{x} B\xrightarrow{y} C\dashrightarrow$$
and any morphism $c\in{\rm Hom}_{\B}(P,C)$, there exists $b\in{\rm Hom}_{\B}(P,B)$ satisfying $y\circ b=c$. We denote by $\mathcal P$ the full subcategory of projective objects in $\B$. Dually, we can define \emph{injective} objects and denote by $\mathcal I$ the full subcategory of injective objects in $\B$.

\begin{defn}
We say that $\B$ {\it has enough projectives}, if
for any object $C\in\B$, there exists an $\EE$-triangle
$$A\xrightarrow{x} P\xrightarrow{y} C\dashrightarrow$$
with $P\in\mathcal P$.  We can define the notion of \emph{having enough injectives} dually.
\end{defn}

Throughout the rest of the paper, let $k$ be a field and  $\B$ be a Krull-Schmidt, $k$-linear, Hom-finite  extriangulated category with enough projectives and enough injectives. Moreover, we assume that $\B$ satisfies \cite[Condition 5.8]{NP}, i.e.,
\begin{description}
\item[(WIC)] for any morphisms $f:A\to B$ and $g:B\to C$ in $\B$, if $g\circ f$ is an inflation, then so is $f$; if $g\circ f$ is a deflation, then so is $g$.
\end{description}
 All triangulated categories and Krull-Schmidt, Hom-finite, $k$-linear exact categories satisfy this condition.

When we say that $\C$ is a subcategory of $\B$, we always assume that $\C$ is full and closed under isomorphisms, direct sums and direct summands. Note that we do not assume any subcategory we construct has such property.

Let $\C$ and $\D$ be subcategories of $\B$. We denote the following full subcategories
$$\Cone(\C,\D)=\{X\in \B \text{ }|\ \text{there exists an $\EE$-triangle }  C\to D\to X\dashrightarrow ~\mbox{with}~C\in \C,\ D\in \D\},$$
$$\CoCone(\C,\D)=\{X\in \B \text{ }|\ \text{there exists an $\EE$-triangle }  X\to C\to D\dashrightarrow ~\mbox{with}~C\in \C,\ D\in \D\}.$$

\begin{defn}\label{Cone}
A subcategory $\C$ of $\B$ is called \emph{closed under cones} (resp. \emph{closed under cocones}) if $\Cone(\C,\C)\subseteq\C$ (resp. $\CoCone(\C,\C)\subseteq\C$).
\end{defn}

For any subcategory $\C$ of $\B$, let $\C^{\vee}_0=\C^{\wedge}_0=\C$. We denote the full subcategories $\C^{\wedge}_i$ and $\C^{\vee}_i$ for $i>0$ inductively by
$$\C^{\wedge}_i=\Cone(\C^{\wedge}_{i-1},\C),\ \C^{\vee}_i=\CoCone(\C,\C^{\vee}_{i-1}),$$
and denote the full subcategories
$$\C^{\wedge}=\bigcup\limits_{i\geq 0}\C^{\wedge}_i,\ \C^{\vee}=\bigcup\limits_{i\geq 0}\C^{\vee}_i.$$

For any subcategory $\C$ of $\B$, let $\Omega^0 \C=\C$. Then we can define $\Omega^i \C$ for any $i\geq 1$ inductively by $\Omega^i \C=\CoCone(\mathcal P, \Omega^{i-1}\C)$. Dually, let $\Sigma^0 \C=\C$. Then we can define $\Sigma^i \C$ for any $i\geq 1$ inductively by $\Sigma^i \C=\Cone(\Sigma^{i-1}\C,\mathcal I)$.

For any $X\in\B$, we write $\Omega X$ (resp. $\Sigma X$) as an object in $\B$ (which is not necessarily unique) sitting in an $\EE$-triangle
$$\text{$\Omega X\to P\to X\dashrightarrow$ (resp. $X\to I\to \Sigma X\dashrightarrow$)}$$
with $P\in \mathcal P$ (resp. $I\in \mathcal I$). Similarly, one can define $\Omega^i X$ and $\Sigma ^iX$ for any $i\geq 1$ inductively by
$$\Omega^i X=\Omega(\Omega^{i-1}X) \text{ and }\Sigma^iX=\Sigma(\Sigma^{i-1}X).$$
Obviously we have $\Omega^i X\in \Omega^i \C$ (resp. $\Sigma^i X\in \Sigma^i \C$) if $X\in \C$.

\begin{defn}
For any objects $X,Y$ of $\B$, we define \emph{higher extension groups} as $$\EE^{i+1}(X,Y):=\EE(\Omega^iX,Y), i\geq 1.$$
\end{defn}

The following lemma tells us that the definition of higher extensions does not depend on the choice of $\Omega^iX$.

\begin{lem}[{\cite[Proposition 5.2]{LN}}]\label{lem:high}
For any objects $X,Y$ of $\B$ and any non-negative integers $i,j$, we have the following isomorphisms of groups
$$\EE(\Omega^{i+j}X,  Y)\simeq \EE(\Omega^iX,\Sigma^jY)\simeq \EE(X,\Sigma^{i+j} Y).$$
\end{lem}

As in triangulated categories  and exact categories, $\EE$-triangles can produce long exact sequences of  abelian groups.

\begin{lem}[{\cite[Propositin~5.2]{LN}}]\label{lem:long}
For any $\EE$-triangle $A\xrightarrow{x}B\xrightarrow{y}C\stackrel{\delta}{\dashrightarrow}$ and any objects $X$ and $Y$ of $\B$, there are  long exact sequences
$$\begin{array}{cccccccc}
&&\Hom_{\B}(X,A)&\xrightarrow{x\circ-}&\Hom_{\B}(X,B)&\xrightarrow{y\circ-}&\Hom_{\B}(X,C)\\
&\to&\EE(X,A)&\to&\EE(X,B)&\to&\EE(X,C)\\
&\cdots&\cdots&\cdots&\cdots&\cdots&\cdots\\
&\to&\EE^i(X,A)&\to&\EE^i(X,B)&\to&\EE^i(X,C)\\
&\to&\EE^{i+1}(X,A)&\to&\cdots&\cdots&\cdots
\end{array}$$
and
$$\begin{array}{cccccccc}
&&\Hom_{\B}(C,Y)&\xrightarrow{-\circ y}&\Hom_{\B}(B,Y)&\xrightarrow{-\circ x}&\Hom_{\B}(A,Y)\\
&\to&\EE(C,Y)&\to&\EE(B,Y)&\to&\EE(A,Y)\\
&\cdots&\cdots&\cdots&\cdots&\cdots&\cdots\\
&\to&\EE^i(C,Y)&\to&\EE^i(B,Y)&\to&\EE^i(A,Y)\\
&\to&\EE^{i+1}(C,Y)&\to&\cdots&\cdots&\cdots.
\end{array}$$
\end{lem}

\subsection{Presilting subcategories and cotorsion pairs}

We recall the definition of a presilting subcategory in an extriangulated category.

\begin{defn}[{\cite[Section~3]{AT}}]\label{presil}
A subcategory $\s$ of $\B$ is called a \emph{presilting subcategory} if $\EE^i(\s,\s)=0, \forall i\geq 1$. An object $S$ of $\B$ is called a \emph{presilting object} if its additive closure $\add S$ is a presilting subcategory.
\end{defn}

We have the following properties of $\s^\wedge$ and $\s^\vee$ for a presilting subcategory $\s$.

\begin{lem}[{\cite[Section 3]{AT}}]\label{lem1}
Let $\s$ be a presilting subcategory. Then the following statements hold.
\begin{enumerate}
\item $\EE^i(\s^{\vee},\s^{\wedge})=0, \forall i>0$ and $\s^{\vee}\cap \s^{\wedge}=\s$.
\item $\s^{\vee}$ is closed under direct sums, direct summands, extensions and cocones.
\item $\s^{\wedge}$ is closed under direct sums, direct summands, extensions and cones.
\item $\s^\perp=(\s^{\vee})^\perp$ and ${}^\perp\s={}^\perp(\s^{\wedge})$.
\item $\s^\perp\cap\s^\vee=\s={}^\perp\s\cap\s^{\wedge}$.
\end{enumerate}
\end{lem}

We recall the definition of a cotorsion pair in an extriangulated category. Note that it is also called a \emph{complete} cotorsion pair in many articles (e.g., \cite{S}).

\begin{defn}[{\cite[Definition 4.1]{NP}}]\label{cotorsion}
Let $\U$ and $\V$ be two subcategories of $\B$. We call $(\U,\V)$ a \emph{cotorsion pair} if $$\EE(\U,\V)=0\text{ and }\B=\Cone(\V,\U)=\CoCone(\V,\U).$$
\end{defn}

Let $\C$ be a subcategory of $\B$. We introduce the following notions:
\begin{enumerate}
\item $\C^{\bot}=\{X\in \B\text{ }|\text{ }\EE^i(\C,X)=0, \forall i>0\}$;
\item $\C^{\bot_1}=\{X\in \B\text{ }|\text{ }\EE(\C,X)=0\}$;
\item ${^{\bot}}\C=\{X\in \B\text{ }|\text{ }\EE^i(X,\C)=0, \forall i>0\}$;
\item ${^{\bot_1}}\C=\{X\in \B\text{ }|\text{ }\EE(X,\C)=0\}$.
\end{enumerate}

%Let $\C$ be a subcategory of $\B$. A \emph{right $\C$-approximation} of an object $X\in\B$ is a morphism $Y\to X$ with $Y\in\C$ such that for any morphism $g:Z\to X$ with $Z\in\C$, there is a morphism $h:Z\to Y$ such that $g=f\circ h$. We call $\C$ \emph{contravariantly finite} provided that for any object $X$ of $\B$, there is a right $\C$-approximation of $X$. The notion of \emph{left approximation} and \emph{covariantly finite} are defined dually.

%Recall that a morphism $f:Y\to X$ is \emph{right minimal} if each morphism $g:Y\to Y$ such that $f\circ g=f$ is an isomorphism. A right $\C$-approximation of an object $X$ is \emph{minimal} if it is a right minimal morphism. The notion of \emph{left minimal} and \emph{minimal left approximation} are defined dually.

We have the following properties for cotorsion pairs.

\begin{lem}[{\cite[Remark~3.2]{CZZ}}]\label{lem:basic}
Let $(\U,\V)$ be a cotorsion pair in $\B$. Then
\begin{enumerate}
\item $\V=\U^{\bot_1}$;
\item $\U={^{\bot_1}}\V$;
\item $\U$ and $\V$ are closed under extensions;
\item $\mathcal I\subseteq \V$ and $\mathcal P\subseteq \U$;
\item $\U$ is contravariantly finite in $\B$ and $\V$ is covariantly finite in $\B$.
\end{enumerate}
\end{lem}

\subsection{Silting subcategories}

\begin{defn}\label{thick}
A subcategory $\C$ of $\B$ is called \emph{thick} provides that for any $\EE$-triangle $A\to B\to C\dashrightarrow$, if any two objects of
$A,B$ and $C$ belong to $\C$, then so is the third one.
\end{defn}

For any subcategory $\D$ of $\B$, we denote by $\thick \D$ the smallest thick subcategory of $\B$ containing $\D$.

\begin{defn}[{\cite[Definition 5.1]{AT}}]\label{def:sil}
A subcategory $\s$ of $\B$ is called \emph{silting}  if $\s$ is presilting and $\B=\thick\s$.
\end{defn}

Note that for any presilting subcategory $\s$, since $\thick \s$ is closed under extensions, it is an extriangulated category. Moreover, inheriting from $\B$, it is Krull-Schmidt, Hom-finite and $k$-linear. By definition, we have the following lemma.

\begin{lem}\label{main2cor}
Let $\s$ be a presilting subcategory. If $\mathcal P\cup\mathcal I\subseteq\thick\s$, then $\s$ is a silting subcategory in $\thick \s$.
\end{lem}

The following criterion on a presilting subcategory to be silting is a direct consequence of \cite[Theorem~5.5]{AT}.

\begin{thm}\label{main2}
A presilting subcategory $\s$ of $\B$ is silting if and only if $(\s^{\vee},\s^{\wedge})$ is a cotorsion pair.
\end{thm}

\begin{rem}\label{Frob}
If $\B$ is a Frobenius (i.e. $\mathcal P=\mathcal I$) extriangulated category, by \cite[Corollary~7.4 and Remark~7.5]{NP}, $\B/[\mathcal P]$ is a triangulated category. In this case, by Lemma~\ref{lem1}, any silting subcategory of $\B$ contains $\mathcal P$. By \cite[Theorems~1.1 and 4.6]{MDZH}, $\s$ is a silting subcategory in $\B$ if and only if $\s/[\mathcal P]$ is a silting subcategory in $\B/[\mathcal P]$.
\end{rem}

We give some examples of silting subcategories.

\begin{exm}
\begin{itemize}
\item[(1)] If $\B$ is an abelian category, then any silting subcategory in $\B$ is a silting subcategory in $\mathrm{D}^b(\B)$. This is because on one hand, $\Ext^i_{\B}(\s,\s)=0$ implies $\Hom_{\mathrm{D}^b(\B)}(\s,\s[i])=0$ for any $i>0$; on the other hand, $\mathrm{D}^b(\B)=\thick\B=\thick\s$.
\item[(2)] Let $\Lambda$ be a finite dimensional $k$-algebra and $\mod \Lambda$ be the category of finitely generated $\Lambda$-modules. If $\gl \Lambda<\infty$, then the $n$-tilting modules in $\mod \Lambda$ are the silting objects (see Theorem~\ref{main4.4} for more general results). If $\gl \Lambda=\infty$, then by \cite[Example 2.5]{AI}, there is no silting subcategory in $\mathrm{D}^b(\Lambda)$, hence by (1), $\mod \Lambda$ has no silting subcategory.
\item[(3)] Let $A=\bigoplus_{i\geq 0}A_i$ be a finite dimensional positively graded $1$-Gorenstein algebra with $\gl A_0<\infty$. By \cite[Theorem 3.0.3]{LZ}, there is a silting object $T:=\bigoplus\limits_{i\geq 0}A(i)_{\leq 0}$ in $\mathrm{D}_{sg}(\mod^{\mathbb{Z}}A)$, where $(i)$ is the degree shift functor. Since $\mathrm{D}_{sg}(\mod^{\mathbb{Z}}A)\simeq \underline {\mathrm{CM}}^{\mathbb{Z}}(A)$, we have a silting object in $\underline {\mathrm{CM}}^{\mathbb{Z}}(A)$. By the discussion in Remark \ref{Frob}, we get a silting object in ${\mathrm{CM}}^{\mathbb{Z}}(A)$.

\end{itemize}
\end{exm}

\section{Silting reduction}\label{sec:red}

In this section, let $\W\subseteq \B$ be a presilting subcategory. We assume that $\W$ satisfies the following condition:
\begin{enumerate}
\item[(CP)] $(\W^{\vee},\W^{\bot})$ and $({^{\bot}}\W,\W^{\wedge})$ are cotorsion pairs in $\B$.
\end{enumerate}

\begin{rem}\label{rmk1}
By Lemma~\ref{lem:basic}~(4), condition (CP) implies $\mathcal P\subseteq\W^{\vee}$ and $\mathcal I\subseteq\W^{\wedge}$. So in this case, we have $\thick\mathcal P\subseteq \thick \W^{\vee}=\thick\W$ and $\thick\mathcal I\subseteq\thick\W^{\wedge}=\thick\W$.
\end{rem}

\begin{defn}
Let $X$ be an object of $\B$. The \emph{projective dimension} of $X$ is defined to be
$$\pd X:=\sup\{n\mid \EE^n(X,Y)\neq 0\text{ for some object $Y$ of $\B$}\}.$$
For any subcategory $\C$ of $\B$, the \emph{projective dimension} of $\C$ is defined to be
$$\pd\C:=\sup\{\pd X\mid X\in\C\}.$$
Dually, we can define the \emph{injective dimensions} $\id X$ and $\id\C$.
\end{defn}

\begin{rem}
If $\W$ is silting, condition ($\operatorname{CP}$) holds. If $\W$ is not silting, condition ($\operatorname{CP}$) implies that $\pd \B=\infty$ and $\id \B=\infty$. This is because if $\pd \B<\infty$ or $\id \B<\infty$, by Remark~\ref{rmk1}, we have either $\B=\thick \mathcal P\subseteq\thick \W$ or $\B=\thick \mathcal I\subseteq \thick \W$. So $\W$ is a silting subcategory, a contradiction.
\end{rem}

\begin{rem}
When $\B$ is a triangulated category, condition (CP) is equivalent to the conditions in \cite[Theorem~1.1]{IY} by \cite[Proposition~3.2]{IY}.
\end{rem}

\subsection{Reduction via subfactor categories}

Denote by $[\W](A,B)$ the subgroup of $\Hom_{\B}(A,B)$ consisting of the morphisms $f$ factoring through an object in $\W$. We denote by $\B/[\W]$ (or $\overline \B$ for short) the category which has the same objects as $\B$, and
$$\Hom_{\oB}(A,B)=\Hom_{\B}(A,B)/[\W](A,B)$$
for $A,B\in \B$. For any morphism $f\in \Hom_{\B}(A,B)$, we denote its image in $\Hom_{\oB}(A,B)$ by $\overline f$.

Let $\Z={^{\bot}}\W \cap\W^{\bot}$. Then $\Z$ is closed under extensions, and hence is an extriangulated subcategory in $\B$, whose $\EE$-triangles are those in $\B$ such that all terms are in $\Z$, see \cite[Remark 2.18]{NP}.

\begin{lem}\label{lem:Fro}
$\Z$ is a Frobenius extriangulated subcategory (that is, $\Z$ has enough projectives and enough injectives, and its projectives coincide with its injectives) in which $\W$ is the subcategory of projective-injective objects, which implies that $\Z/[\W]$ is a triangulated category.
\end{lem}

\begin{proof}
Since $(\W^{\vee},\W^{\bot})$ and $({^{\bot}}\W,\W^{\wedge})$ are cotorsion pairs in $\B$, we have $\W=\W^{\vee}\cap\W^{\bot}={^{\bot}}\W\cap\W^{\wedge}$. Since $\EE(\W,\Z)=0$ and $\EE(\Z,\W)=0$, $\W$ is a subcategory of projective-injective objects in $\Z$. Let $Z\in \Z$. It admits an $\EE$-triangle
$$Z\to V\to Y\dashrightarrow$$
with $V\in \W^{\wedge}$ and $Y\in {^{\bot}}\W$. Since ${}^\perp\W$ is closed under extensions, we have $V\in {^{\bot}}\W\cap\W^{\wedge}=\W$.  By Lemma \ref{lem1} (4), we have $\EE^i(\W^{\vee},\W^\perp)=0$ for any $i>0$. Applying Lemma~\ref{lem:long} to the above $\EE$-triangle, we have an exact sequence
$$0=\EE(\W^{\vee},V)\to \EE(\W^{\vee},Y)\to \EE^2(\W^{\vee},Z)=0,$$
which implies $\EE(\W^{\vee},Y)=0$. Thus $Y\in \Z$. Hence $\Z$ has enough injectives $\W$. Dually, one can show that $\Z$ has enough projectives $\W$. So $\Z$ is a Frobenius category.

\end{proof}

By \cite[Theorem 3.13]{ZZ}, we have the following triangulated structure on $\Z/[\W]$.

\begin{lem}
$\Z/[\W]$ is a triangulated category, whose suspension functor $$\langle1\rangle: A\mapsto A\langle1\rangle,\ a\mapsto a\langle1\rangle$$ and distinguished triangles $$A\xrightarrow{\overline f} B\xrightarrow{\overline g} C\xrightarrow{\overline h} A\langle 1\rangle$$ are given by the following commutative diagram:
$$\xymatrix@C=0.8cm@R0.7cm{
A \ar[r]^f \ar@{=}[d] &B \ar[r]^g \ar[d] &C \ar[d]^h \ar@{-->}[r] &\\
A \ar[r] \ar[d]_a &W_A \ar[r] \ar[d] &A\langle 1\rangle \ar[d]^{a\langle 1\rangle}  \ar@{-->}[r] &\\
D \ar[r] &W_D \ar[r] &D\langle 1\rangle \ar@{-->}[r] &
}
$$
with $W_A,W_D\in \W$. Here $A\xrightarrow{f} B\xrightarrow{g} C\dashrightarrow$ is an arbitrary $\EE$-triangle in $\Z$ and $a:A\to D$ is an arbitrary morphism in $\Z$.
\end{lem}

The following lemma is useful. The proof is left to the readers.

\begin{lem}\label{lem:iso}
$\Hom_{\oB}({}^\perp\W,\W^{\wedge})=0$ and $\Hom_{\oB}(\W^{\vee},\W^\perp)=0$.
\end{lem}

For any objects $X,Y\in \Z$, we denote by $\EE^i_\Z(X,Y)$ the $i$-th extension group of $X,Y$ in the extriangulated category $\Z$.

\begin{lem}\label{lem:formula}
For any $X,Y\in\Z$, we have the isomorphisms
$$\EE^{i}_\Z(X,Y)\cong \EE^{i}(X,Y),\ i\geq 1.$$
\end{lem}

\begin{proof}
By Lemma~\ref{lem:Fro}, there are $\EE$-triangles in $\Z$:
$$X_j\to W_j\to X_{j-1}\dashrightarrow,\ j\geq 1$$
with all $W_j\in\W$ and $X_0=X$. By definition, we have $\EE_\Z^{i}(X,Y)=\EE_\Z(X_{i-1},Y)=\EE(X_{i-1},Y)$ for any $i\geq 1$. When $i=1$, we already have the required formula. Suppose that $i>1$. Applying Lemma~\ref{lem:long} to the above triangle, we have the following exact sequences
$$0=\EE^{i-j}(W_j,Y)\to\EE^{i-j}(X_j,Y)\to\EE^{i-j+1}(X_{j-1},Y)\to\EE^{i-j+1}(W_j,Y)=0,\ 1\leq j< i.$$
So we have isomorphisms $\EE^i(X_0,Y)\cong\EE^{i-1}(X_1,Y)\cong\EE^{i-2}(X_2,Y)\cong\cdots\cong\EE(X_{i-1},Y)$. Hence we have $\EE^{i}_\Z(X,Y)\cong \EE^{i}(X,Y)$ as required.
\end{proof}

As a direct consequence of the formulas in Lemma~\ref{lem:formula}, we have the following correspondence between presilting subcategories of $\B$ containing $\W$ and presilting subcategories of $\Z$.

\begin{lem}\label{psilz}
Let $\s\supseteq \W$. Then $\s$ is a presilting subcategory in $\B$ if and only if $\s$ is a presilting subcategory in $\Z$.
\end{lem}

We prove that this correspondence can be restricted to silting subcategories.

\begin{prop}\label{silz}
Let $\s\supseteq \W$ be a presilting subcategory. Then $\s$ is a silting subcategory if and only if it is also a silting subcategory in $\Z$.
\end{prop}

We divide the proof into two parts.

\begin{lem}
Let $\s\supseteq \W$ be a presilting subcategory. Then if $\s$ is a silting subcategory in $\Z$, it is also a silting subcategory in $\B$.
\end{lem}

\begin{proof}
By Lemma~\ref{psilz}, we only need to show that if $\s$ is a silting subcategory in $\Z$, then $\thick\s=\B$.

Since $\s$ is a silting subcategory in $\Z$, by Theorem~\ref{main2}, there is a cotorsion pair $(\U,\V)$ in $\Z$ such that $\U\subseteq \s^{\vee}$ and $\V\subseteq \s^{\wedge}$. For any object $B\in \B$, it admits a commutative diagram
$$\xymatrix@C=0.8cm@R0.7cm{
V_B \ar[r] \ar@{=}[d] &T_B\ar[r] \ar[d] &B \ar[d] \ar@{-->}[r] &\\
V_B \ar[r] &Z^B \ar[r] \ar[d] &C \ar[d] \ar@{-->}[r] &\\
&U^B \ar@{=}[r] \ar@{-->}[d]  &U^B \ar@{-->}[d]\\
&&&
}
$$
where $T_B\in {^{\bot}}\W$, $V_B\in \W^{\wedge}$, $U^B\in \W^{\vee}$ and $Z^B\in \Z$. Since $Z^B$ admits an $\EE$-triangle $$V_Z'\to U_Z'\to Z^B\dashrightarrow$$ where $V_Z'\in \V$ and $U_Z'\in \U$, we have a commutative diagram
$$\xymatrix@C=0.8cm@R0.7cm{
V_Z' \ar@{=}[r] \ar[d] &V_Z' \ar[d]\\
U_T \ar[r] \ar[d] &U_Z' \ar[d] \ar[r] &U^B \ar@{=}[d] \ar@{-->}[r] &\\
T_B \ar[r] \ar@{-->}[d] &Z^B \ar[r] \ar@{-->}[d] &U^B \ar@{-->}[r] &\\
&&&
}
$$
with $U_T\in \s^{\vee}$. Then we can get a commutative diagram
$$\xymatrix@C=0.8cm@R0.7cm{
V_Z' \ar@{=}[r] \ar[d] &V_Z' \ar[d]\\
V_T \ar[r] \ar[d] &U_T \ar[r] \ar[d] &B\ar@{=}[d]  \ar@{-->}[r] &\\
V_B \ar[r] \ar@{-->}[d] &T_B\ar[r] \ar@{-->}[d] &B \ar@{-->}[r] &\\
&&&
}
$$
with $V_T\in \s^{\wedge}$. Hence $B\in \thick \s$.
\end{proof}

\begin{lem}
Let $\s\supseteq \W$ be a presilting subcategory. Then if $\s$ is a silting subcategory in $\B$, it is also a silting subcategory in $\Z$.
\end{lem}

\begin{proof}
By Theorem~\ref{main2}, if $\s$ is a silting subcategory in $\B$, then $\s$ admits a cotorsion pair $(\s^{\vee},\s^{\wedge})$. Denote $\s^{\wedge}\cap \Z$ by $\V_\Z$.

We first show the following fact:

For an $\EE$-triangle $V_Z\to S\to Z\dashrightarrow$, if $Z\in \Z$, $S\in \s$ and $V_Z\in \s^{\wedge}$, then $V_Z\in  \V_\Z$.

$Z$ admits an $\EE$-triangle $Z'\to W\to Z\dashrightarrow$ where $W\in \W$ and $Z'\in \Z$, since $\EE(W,V_Z)=0$, we have the following commutative  diagram.
$$\xymatrix@C=0.8cm@R0.7cm{
&Z' \ar@{=}[r] \ar[d] &Z' \ar[d]\\
V_Z \ar[r] \ar@{=}[d] &V_Z\oplus W \ar[r] \ar[d] &W \ar[d] \ar@{-->}[r] &\\
V_Z\ar[r] &S \ar[r] \ar@{-->}[d] &Z \ar@{-->}[d] \ar@{-->}[r] &\\
&&&
}
$$
Since $\Z$ is closed under direct summands and extensions, we can get that $V_Z\in \Z$. Hence $V_Z\in \V_\Z$.

The argument above shows that any object $Y\in \V_\Z$ admits the following $\EE$-triangles:
$$L_0\to S_0\to Y\dashrightarrow, \quad L_1\to S_1\to L_0\dashrightarrow, \quad \cdots, \quad S_n \to S_{n-1} \to L_{n-1}\dashrightarrow$$
where $S_i \in \s, i=0,1,...,n$ and $L_j\in \V_\Z,  j=1,...,n-1$.

We show that any object $Z\in \Z$ can be generated by $\s$, which implies that $\s$ is a silting subcategory in $\Z$ by definition. From the argument above we can get that any object in $\V_\Z$ can be generated by $\s$.  Since $Z$ admits an $\EE$-triangle $V_Z\to U_Z\to Z\dashrightarrow$ where $U_Z\in \s^{\vee}$ and $V_Z\in \s^{\wedge}$, we can assume that we have shown that:
$$\text{If }\U\in \s^{\vee}_k,\text{ then }\Z\text{ can be generated by }\s.$$
Now let $U_Z\in \s^{\vee}_{k+1}$. Since $U_Z$ admits an $\EE$-triangle $U_Z\to S_0\to U_Z^0\dashrightarrow$ where $S_0\in \s$ and $U_Z^0\in \s^{\vee}_k$, we have the following commutative diagrams
$$\xymatrix@C=0.8cm@R0.7cm{
V_Z \ar[r] \ar@{=}[d] &U_Z \ar[r] \ar[d] &Z \ar[d] \ar@{-->}[r] &\\
V_Z \ar[r] &S_0 \ar[r] \ar[d] &V_Z^0 \ar[d]  \ar@{-->}[r] &\\
&U_Z^0 \ar@{=}[r] \ar@{-->}[d] &U_Z^0 \ar@{-->}[d]\\
&&&\\
}\quad \xymatrix@C=0.8cm@R0.7cm{
Z \ar[r] \ar[d] &W \ar[d] \ar[r] &Z \langle 1\rangle \ar@{-->}[r] \ar@{=}[d] &\\
V_Z^0 \ar[d] \ar[r] &W\oplus U_Z^0 \ar[r] \ar[d] &Z  \langle 1\rangle \ar@{-->}[r] &\\
U_Z^0 \ar@{-->}[d] \ar@{=}[r] &U_Z^0 \ar@{-->}[d]\\
&&&
}
$$
where $W\in \W$ and $V_Z^0\in \s^{\wedge}$. Since $W\oplus U_Z^0\in \s^{\vee}_k$, by hypothesis, we can get that $Z \langle 1\rangle $ can be generated by $\s$. Hence $Z$ can be generated by $\s$.

\end{proof}

According to \cite[Theorem 1.1]{MDZH}, we can get a bijection between the silting subcategories in $\Z$ and the silting subcategories in $\Z/[\W]$. Hence we have the following corollary.

\begin{cor}\label{silzw}
Let $\s\supseteq \W$ be a presilting subcategory. Then $\overline \s$ is a silting subcategory in $\Z/[\W]$ if and only if it is a silting subcategory in $\B$.
\end{cor}

\subsection{A localization of $\B$ realized by $\Z/[\W]$}

 In the rest of this section, we assume that $\B$ is skeletally small. Denote by $\R$ the following class of morphisms
$$\{f:A\to B \mid \text{ there is an } \EE\mbox{-triangle} \text{ } A\xrightarrow{f} B\to M\dashrightarrow \text{ with } M\in \thick \W\}.$$
Denote by $\B/(\thick \W)$ the Gabriel-Zisman localization of $\B$ with respect to $\R$
(see \cite{GZ} for more details of such localization). In the localization $\B/(\thick \W)$, any morphism $f\in \R$ becomes invertible and any object in $\thick \W$ becomes zero. For any morphism $g$, we denote its image in $\B/(\thick \W)$ by $\underline g$.

We will show the following theorem.

\begin{thm}\label{main5}
The Gabriel-Zisman localization $\B/(\thick \W)$ is equivalent to $\Z/[\W]$.
\end{thm}

We first establish a functor $G$ from $\B$ to $\Z/[\W]$. For any object $B$ of $\B$, we fix two $\EE$-triangles
\begin{align*}
B\xrightarrow{b^+} B^+\to U_B\dashrightarrow,\quad
V^B\to Z_B\xrightarrow{z_B} B^+\dashrightarrow
\end{align*}
where $U_B\in \W^{\vee}$, $V^B\in \W^{\wedge}$, $b^+$ is a minimal left $(\W^{\bot})$-approximation and $z_B$ is a minimal right $({^{\bot}}\W)$-approximation. We have $Z_B\in \Z$. We shall use these facts frequently later without mentioned.

Let $f:B\to C$ be any morphism in $\B$. Then we have the following commutative diagrams of $\EE$-triangles:
$$\xymatrix@C=0.8cm@R0.7cm{
B \ar[r]^{b^+} \ar[d]^f &B^+ \ar[r] \ar[d]^{f^+} &U_B \ar[d] \ar@{-->}[r] &\\
C \ar[r]^{c^+} &C^+ \ar[r] &U_C \ar@{-->}[r] &,\\
}\quad
\xymatrix@C=0.8cm@R0.7cm{
V^B \ar[r] \ar[d] &Z_B \ar[r]^{z_B} \ar[d]^{z_f} &B^+ \ar[d]^{f^+} \ar@{-->}[r] &\\
V^C \ar[r] &Z_C  \ar[r]^{z_C} &C^+ \ar@{-->}[r] &.
}
$$

\begin{lem}\label{unique}
For any morphism $f:B\to C$, $\overline z_f$ is unique.
\end{lem}

\begin{proof}
If we have $(f^+)':B^+\to C^+$ and $z_f':Z_B\to Z_C$ such that the following diagrams commute:
$$\xymatrix@C=0.8cm@R0.7cm{
B \ar[r]^{b^+} \ar[d]^f &B^+ \ar[r] \ar[d]^{(f^+)'} &U_B \ar[d] \ar@{-->}[r] &\\
C \ar[r]^{c^+} &C^+ \ar[r] &U_C \ar@{-->}[r] &,\\
} \quad
\xymatrix@C=0.8cm@R0.7cm{
V^B \ar[r] \ar[d] &Z_B \ar[r]^{z_B} \ar[d]^{z_f'} &B^+ \ar[d]^{(f^+)'} \ar@{-->}[r] &\\
V^C \ar[r] &Z_C  \ar[r]^{z_C} &C^+ \ar@{-->}[r] &,
}
$$
then $f^+-(f^+)':B^+\to C^+$ factors through $U_B$, hence by Lemma~\ref{lem:iso} factors through $\W$. Then $z_C\circ(z_f-z_f')$ factors through an object $W\in \W$. Let $z_C\circ(z_f-z_f')=w_2\circ w_1$ for morphisms $w_1:Z_B\to W$ and $w_2:W\to C^+$. Then there is a morphism $w_3:W\to Z_C$ such that $w_2=z_C\circ w_3$. Thus $z_C\circ\left((z_f-z_f')-w_3w_1\right)=0$, then $(z_f-z_f')-w_3w_1$ factors through $V^C$, hence by Lemma~\ref{lem:iso} factors through $\W$. Then $\overline z_f=\overline z_f'$.
\end{proof}

By the proof of this lemma, we can get the following corollary immediately.

\begin{cor}
Let $f,f'\in \Hom_\B(B,C)$ such that $\overline f=\overline f'$. Then $\overline z_f=\overline z_{f'}$.
\end{cor}

Now we can define a functor $G:\B\to \Z/[\W]$ as following:
$$G(B)=Z_B, \quad G(f)=\overline z_f.$$

\begin{rem}\label{rem1}
By the construction  of $G$, we have $G(b^+)=G(z_B)=\overline 1_{Z_B}$.
\end{rem}

We show a very important property of this functor.

\begin{prop}\label{im}
$G(f)$ is an isomorphism for any morphism $f:B\to C$ in $\R$.
\end{prop}

\begin{proof}
Since $f\in\R$, there is an $\EE$-triangle $B\xrightarrow{f} C\to M\dashrightarrow$ with $M\in
\thick \W$. Then we have the following commutative diagram.
$$\xymatrix@C=0.8cm@R0.7cm{
B \ar[r]^{b^+} \ar[d]^f &B^+ \ar[r]^{u_B} \ar[d]^{d_1} &U_B \ar@{=}[d] \ar@{-->}[r] &\\
C \ar[r]^{c'} \ar[d] &C' \ar[d]^{d'_2} \ar[r] &U_B \ar@{-->}[r] &\\
M \ar@{=}[r] \ar@{-->}[d] &M \ar@{-->}[d]\\
&&&
}
$$
Since $C^+\in \W^{\bot}$, we have $\EE(U_B,C^+)=0$. There is a morphism $d_2:C'\to C^+$ such that $d_2c'=c^+$. Then we have the following commutative diagram
$$\xymatrix@C=0.8cm@R0.7cm{
B \ar[r]^{b^+} \ar[d]^f &B^+ \ar[r]^{u_B} \ar[d]^{d_1} &U_B \ar@{=}[d] \ar@{-->}[r] &\\
C \ar[r]^{c'} \ar@{=}[d] &C' \ar[d]^{d_2} \ar[r]^{u'} &U_B \ar[d] \ar@{-->}[r] &\\
C \ar[r]^{c^+} &C^+ \ar[r] &U_C \ar@{-->}[r] &
}
$$
such that $\overline {d_2d_1}=\overline f^+$. By Proposition~\ref{prop:NPLN}, this diagram induces an $\EE$-triangle $C'\xrightarrow{\svecv{d_2}{{\red u'}}} C^+\oplus U_B\to U_C \dashrightarrow$. Then we have a commutative diagram
$$\xymatrix@C=0.8cm@R0.7cm{
B^+ \ar[r]^{d_1} \ar@{=}[d] &C' \ar[r]^{d_2'} \ar[d]^-{\svecv{d_2}{u'}} &M \ar@{-->}[r] \ar[d] &\\
B^+ \ar[r] &C^+\oplus U_B \ar[r] \ar[d] &M_1 \ar[d] \ar@{-->}[r]&\\
&U_C \ar@{=}[r] \ar@{-->}[d] &U_C \ar@{-->}[d]\\
&&&
}$$
with $M_1\in \thick \W$. Since $U_B$ admits an $\EE$-triangle $U_B\xrightarrow{w} W\to U_1\dashrightarrow$ where $W\in \W$ and $U_1\in \W^{\vee}$, we have the following commutative diagram
$$\xymatrix@C=0.8cm@R0.7cm{
B^+ \ar[r]^-{\svecv{d_2d_1}{ u_B}} \ar@{=}[d] &C^+\oplus U_B \ar[r] \ar[d]^-{\left(\begin{smallmatrix}
1& 0\\
0& w
\end{smallmatrix}\right)} &M_1 \ar[d] \ar@{-->}[r]&\\
B^+ \ar[r] &C^+\oplus W \ar[r] \ar[d] &V_1 \ar[d] \ar@{-->}[r]&\\
&U_1 \ar@{=}[r] \ar@{-->}[d] &U_1 \ar@{-->}[d]\\
&&&
}
$$
where $V_1\in \W^{\bot}\cap \thick \W=\W^{\wedge}$ by \cite[Lemma 3.13]{AT}. Then we have a commutative diagram
$$\xymatrix@C=0.8cm@R0.7cm{
V^B \ar[r] \ar[d] &V^C \ar@{=}[r] \ar[d] &V^C \ar[d]\\
Z_B \ar[r]^{z_1} \ar[d]_{z_B} &(B^+)' \ar[r]^-{\svecv{z_2}{ w'}} \ar[d] &Z_C\oplus W \ar[r] \ar[d]^-{\left(\begin{smallmatrix}
z_C& 0\\
0& 1
\end{smallmatrix}\right)} &V_1\ar@{=}[d] \ar@{-->}[r]&\\
B^+ \ar@{=}[r] \ar@{-->}[d] &B^+ \ar[r]_-{\svecv{d_2d_1}{ wu_B}} \ar@{-->}[d] &C^+\oplus W \ar[r] \ar@{-->}[d] &V_1 \ar@{-->}[r]&\\
&&&
}
$$
with $(B^+)'\in \W^{\bot}$. We have $z_Cz_2z_1=d_2d_1z_B$. Hence $\overline {z_Cz_2z_1}=\overline {f^+z_B}$, by the proof of Lemma \ref{unique}, we have $\overline z_f=\overline {z_2z_1}$. By Proposition~\ref{prop:NPLN}, this diagram induces an $\EE$-triangle $V^B\to Z_B\oplus V^C\xrightarrow{\svech{z_1}{*}} (B^+)'\dashrightarrow$. We have the following commutative diagram.
$$\xymatrix@C=0.8cm@R0.7cm{
V^B \ar[r] \ar[d] &Z_B\oplus V^C \ar[r]^-{\svech{z_1}{*}} \ar[d]^-{\left(\begin{smallmatrix}
z_1& *\\
*& *
\end{smallmatrix}\right)} &(B^+)' \ar@{=}[d] \ar@{-->}[r]&\\
I \ar[r] \ar[d] &(B^+)'\oplus I \ar[r] \ar[d] &(B^+)' \ar@{-->}[r]&\\
\Sigma V^B \ar@{=}[r] \ar@{-->}[d] &\Sigma V^B \ar@{-->}[d]\\
&&&
}
$$
Since $I$ admits an $\EE$-triangle $V_I\to W_I\to I\dashrightarrow$ where $W_I\in \W$ and $V_I\in \W^{\wedge}$, we have the following commutative diagram
$$\xymatrix@C=0.8cm@R0.7cm{
V_I \ar@{=}[r] \ar[d] &V_I \ar[d]\\
X \ar[r] \ar[d] &(B^+)'\oplus W_I \ar[r] \ar[d] &\Sigma V^B \ar@{=}[d] \ar@{-->}[r]&\\
Z_B\oplus V^C \ar[r]_-{\left(\begin{smallmatrix}
z_1& *\\
*& *
\end{smallmatrix}\right)} \ar@{-->}[d] &(B^+)'\oplus I \ar[r] \ar@{-->}[d] &\Sigma V^B\ar@{-->}[r]&\\
&&&
}
$$
with $X\in \W^{\bot}$. Since $\EE(Z_B,V_I)=0$, $Z_B$ becomes a direct summand of $X$, hence $X$ has the form $Z_B\oplus T$ with $T\in \W^{\bot}$.
Then we have the following commutative diagrams
$$\xymatrix@C=0.8cm@R0.7cm{
V_I \ar@{=}[r] \ar[d] &V_I \ar[d]\\
Z_B\oplus T \ar[r] \ar[d]_-{\left(\begin{smallmatrix}
1& *\\
*& *
\end{smallmatrix}\right)
} &(B^+)'\oplus W_I \ar[r] \ar[d]^-{\left(\begin{smallmatrix}
1& 0\\
0& *
\end{smallmatrix}\right)} &\Sigma V^B \ar@{=}[d] \ar@{-->}[r]&\\
Z_B\oplus V^C \ar[r]_-{\left(\begin{smallmatrix}
z_1& *\\
*& *
\end{smallmatrix}\right)} \ar@{-->}[d] &(B^+)'\oplus I \ar[r] \ar@{-->}[d] &\Sigma V^B\ar@{-->}[r]&\\
&&&
}\quad
\xymatrix@C=0.8cm@R0.7cm{
Z_B\oplus T \ar[r]^-{\left(\begin{smallmatrix}
z_1& *\\
*& *
\end{smallmatrix}\right)} \ar@{=}[d] &(B^+)'\oplus W_I \ar[r] \ar[d]^-{\left(\begin{smallmatrix}
z_2& 0\\
{\red w'}& 0\\
0&1
\end{smallmatrix}\right)} &\Sigma V^B \ar[d] \ar@{-->}[r]&\\
Z_B\oplus T  \ar[r]_-{\left(\begin{smallmatrix}
z_2z_1& *\\
w_1& *\\
w_2&*
\end{smallmatrix}\right)} &Z_C\oplus W\oplus W_I \ar[r] \ar[d] &V \ar[d] \ar@{-->}[r]&\\
&V_1 \ar@{=}[r] \ar@{-->}[d] &V_1 \ar@{-->}[d]\\
&&&
}$$
with $V\in\W^{\wedge}$. Since $V$ admits an $\EE$-triangle $V'\to W'\to V\dashrightarrow$ where $W'\in \W$ and $V'\in \W^{\wedge}$, we have a commutative diagram.
$$\xymatrix@C=0.8cm@R0.7cm{
&V' \ar@{=}[r] \ar[d] &V' \ar[d]\\
Z_B\oplus T \ar[r]^{\alpha} \ar@{=}[d] &Y \ar[r] \ar[d]^-{\beta} &W' \ar[d] \ar@{-->}[r]&\\
Z_B\oplus T \ar[r]_-{\left(\begin{smallmatrix}
z_2z_1& *\\
w_1& *\\
w_2&*
\end{smallmatrix}\right)}  &Z_C\oplus W\oplus W_I \ar[r] \ar@{-->}[d] &V \ar@{-->}[d] \ar@{-->}[r]&\\
&&&
}
$$
%where $\alpha$ is a section and $\beta$ is a retraction
Since $\EE(W',Z_B\oplus T)=0$ and $\EE(Z_C\oplus W\oplus W_I,V')=0$, we have $ Z_B\oplus T\oplus W'\simeq Y\simeq Z_C\oplus W\oplus W_I\oplus V'$. Now we can consider the following commutative diagram
$$\xymatrix@C=0.8cm@R0.7cm{
&&V' \ar@{=}[r] \ar[d] &V' \ar[d]\\
Z_B\oplus T \ar[rr]^-{\left(\begin{smallmatrix}
z_2z_1& *\\
w_1& *\\
w_2&*\\
w_3&*
\end{smallmatrix}\right)=\alpha_1} \ar@{=}[d] &&Z_C\oplus W\oplus W_I\oplus V' \ar[r] \ar[d]^-{\left(\begin{smallmatrix}
1& 0 &0&0\\
0& 1&0&0\\
0&0&1&0
\end{smallmatrix}\right)} &W' \ar[d] \ar@{-->}[r]&\\
Z_B\oplus T \ar[rr]_-{\left(\begin{smallmatrix}
z_2z_1& *\\
w_1& *\\
w_2&*
\end{smallmatrix}\right)}  &&Z_C\oplus W\oplus W_I \ar[r] \ar@{-->}[d] &V \ar@{-->}[d] \ar@{-->}[r]&\\
&&&
}
$$
where $\alpha_1$ is a section. Since $Z_B\oplus T\xrightarrow{\svech{1}{0}} Z_B$ factors through $\alpha_1$, there is a morphism
$$Z_C\oplus W\oplus W_I\oplus V' \xrightarrow{\left(\begin{smallmatrix}
u_0& u_1&u_2 &u_3
\end{smallmatrix}\right)=\gamma} Z_B$$
such that $\svech{1}{0}=\gamma\alpha_1$. Hence $1=u_0z_2z_1+u_1w_1+u_2w_2+u_3w_3$. Since $\overline {u_iw_i}=0,i=1,2,3$, we get $\overline 1=\overline {u_0z_2z_1}=\overline {u_0z_f}$. Thus $\overline z_f$ is a section and $Z_B$ is a direct summand of $Z_C$ in $\Z/[\W]$.
On the other hand, we have the following commutative diagram
$$\xymatrix@C=0.8cm@R0.7cm{
Z_B\oplus V^C \ar[r]^-{\left(\begin{smallmatrix}
z_1& *\\
*& *
\end{smallmatrix}\right)} \ar@{=}[d] &(B^+)'\oplus I \ar[r] \ar[d]^-{\left(\begin{smallmatrix}
z_2& 0\\
 w'& 0\\
0&1
\end{smallmatrix}\right)} &\Sigma V^B \ar[d] \ar@{-->}[r]&\\
Z_B\oplus V^C  \ar[r]_-{\left(\begin{smallmatrix}
z_2z_1& *\\
*& *\\
*&*
\end{smallmatrix}\right)} &Z_C\oplus W\oplus I \ar[r] \ar[d] &V_0 \ar[d] \ar@{-->}[r]&\\
&V_1 \ar@{=}[r] \ar@{-->}[d] &V_1 \ar@{-->}[d]\\
&&&
}$$
where $V_0\in\W^{\wedge}$. Since $V_0$ admits an $\EE$-triangle $V_0'\to W_0'\to V_0\dashrightarrow$ where $W_0'\in \W$ and $V_0'\in \W^{\wedge}$, we have a commutative diagram
$$\xymatrix@C=0.8cm@R0.7cm{
&V_0' \ar@{=}[r] \ar[d] &V_0' \ar[d]\\
Z_B\oplus V^C \ar[r]^-{\alpha'} \ar@{=}[d] &Y' \ar[r] \ar[d]^-{\left(\begin{smallmatrix}
y_1\\
y_2\\
y_3
\end{smallmatrix}\right)} &W_0' \ar[d] \ar@{-->}[r]&\\
Z_B\oplus V^C \ar[r]_-{\left(\begin{smallmatrix}
z_2z_1& *\\
*& *\\
*&*
\end{smallmatrix}\right)}  &Z_C\oplus W\oplus I \ar[r] \ar@{-->}[d] &V_0 \ar@{-->}[d] \ar@{-->}[r]&\\
&&&
}
$$
where $\alpha'$ is a section since $\EE(W_0',Z_B\oplus V^C)=0$. Then $Y'\simeq Z_B\oplus V^C\oplus W_0'$.
Since $\EE(Z_C,V_0')=0$, there is a morphism $y_1':Z_C\to Y'$ such that $y_1y_1'=1$. Hence $Z_C$ is a direct summand of $Y'$, which implies that $Z_C$ is a direct summand of $Z_B$ in $\Z/[\W]$. Hence $Z_B\simeq Z_C$ in $\Z/[\W]$ and $\overline z_f$ is an isomorphism.
\end{proof}

By Proposition \ref{im} and the universal property of the localization functor $L_\R:\B\to \B/(\thick \W)$, we get the following commutative diagram:
$$\xymatrix@C=0.8cm@R0.7cm{
\B\ar[rr]^-{G} \ar[dr]_{L_\R} &&\Z/[\W].\\
&\B/(\thick \W) \ar@{.>}[ur]_{H}
}
$$

The following lemma is important. The proof is an analogue of \cite[Lemma 3.5]{BM}, so we omit it.

\begin{lem}\label{BML}
\begin{itemize}
\item[(1)] Let $X$ be any object in $\B$ and $M\in \thick \W$. Then $X\xrightarrow{\svecv{1_X}{0}} X\oplus M$ is invertible in $\B/(\thick\W)$, its inverse is $X\oplus M\xrightarrow{\svech{1_X}{0}}X$. This implies that $\svech{1_X}{0}$ is also invertible in $\B/(\thick\W)$.
\item[(2)] Let $f,f'\in \Hom_\B(B,C)$. If $\overline f'=0$, then $\underline {f+f'}=\underline f$ in $\B/(\thick\W)$.
\end{itemize}
\end{lem}

This lemma has a useful conclusion.

\begin{cor}\label{BMLcor}
Let $M'\to A\xrightarrow{g} B\dashrightarrow$ be an $\EE$-triangle with $M'\in \thick \W$. Then $\underline g$ is invertible.
\end{cor}

\begin{proof}
We have the following commutative diagram
$$\xymatrix@C=0.8cm@R0.7cm{
M'\ar[r] \ar[d] &A \ar[r]^g \ar[d]^{\svecv{i}{g}} &B \ar@{-->}[r] \ar@{=}[d] &\\
I \ar[r] \ar[d] &I\oplus B \ar[r] \ar[d] &B \ar@{-->}[r] &\\
\Sigma M' \ar@{=}[r] \ar@{-->}[d] &\Sigma M' \ar@{-->}[d]\\
&&
}
$$
with $I\in \mathcal I$. Since $\Sigma M'\in \thick \W$, $A\xrightarrow{\svecv{i}{g}}I\oplus B$ becomes invertible in $\B/(\thick \W)$. By Lemma \ref{BML}, $I\oplus B\xrightarrow{\svech{0}{1}} B$ is invertible in $\B/(\thick \W)$, we can get that $\svech{0}{1}\svecv{i}{g}=g$ is also invertible in $\B/(\thick\W)$.
\end{proof}

\begin{rem}
Let $f:B\to C$ be any morphism in $\R$. In the following commutative diagrams,
$$\xymatrix@C=0.8cm@R0.7cm{
B \ar[r]^{b^+} \ar[d]^f &B^+ \ar[r] \ar[d]^{f^+} &U_B \ar[d] \ar@{-->}[r] &\\
C \ar[r]^{c^+} &C^+ \ar[r] &U_C \ar@{-->}[r] &\\
}\quad
\xymatrix@C=0.8cm@R0.7cm{
V^B \ar[r] \ar[d] &Z_B \ar[r]^{z_B} \ar[d]^{z_f} &B^+ \ar[d]^{f^+} \ar@{-->}[r] &\\
V^C \ar[r] &Z_C  \ar[r]^{z_C} &C^+ \ar@{-->}[r] &
}
$$
$\underline f^{+}$ is invertible. Moreover,  by Corollary \ref{BMLcor}, $\underline z_B$ and $\underline z_C$ are invertible.
\end{rem}

\begin{lem}\label{lem}
Let $f:B\to C$ be any morphism in $\R$. Then we have the following commutative diagram in $\B/(\thick \W)$
$$\xymatrix@C=0.8cm@R0.7cm{
C \ar[r]^{\underline c^+} \ar[d]^{\underline f^{-1}} &C^+ \ar[r]^{\underline z_C^{-1}} \ar[d]^{ (\underline f^+)^{-1}} &Z_C \ar[d]^{\underline z}\\
B \ar[r]_{\underline b^+} &B^+ \ar[r]_{\underline z_B^{-1}} &Z_B
}
$$
where $z:Z_C\to Z_B$ is a morphism in $\Z$ such that $\overline z=G(f)^{-1}$.
\end{lem}

\begin{proof}
By Proposition \ref{im}, $G(f)=\overline z_f$ is invertible, let $\overline z=G(f)^{-1}$. Then $1_{Z_C}-z_fz$ factors through $\W$, by Lemma \ref{BML}, we have $\underline 1_{Z_C}=\underline {z_fz}$.
By the following commutative diagram
$$\xymatrix@C=0.8cm@R0.7cm{
B \ar[r]^{b^+} \ar[d]_f &B^+ \ar[d]^{f^+} &Z_B \ar[l]_{z_B} \ar[d]^{z_f}\\
C \ar[r]_{c^+} &C^+ &Z_C \ar[l]_{z_C}
}
$$
we have $f^{+}z_Bz=z_Cz_fz$. By applying $L_\R$ to this equation, we get $\underline f^{+}\underline z_B\underline z=\underline z_C$, then $\underline z\underline z_C^{-1}=\underline z_B^{-1}(\underline f^+)^{-1} $. Hence we have the following commutative diagram.
$$\xymatrix@C=0.8cm@R0.7cm{
C \ar[r]^{\underline c^+} \ar[d]^{\underline f^{-1}} &C^+ \ar[r]^{\underline z_C^{-1}} \ar[d]^{ (\underline f^+)^{-1}} &Z_C \ar[d]^{\underline z}\\
B \ar[r]_{\underline b^+} &B^+ \ar[r]_{\underline z_B^{-1}} &Z_B
}
$$
\end{proof}

We now give the proof of Theorem \ref{main5}.

\begin{proof}[Proof of Theorem~\ref{main5}.]
We show that $H$ is an equivalence. Since $G|_{\Z}$ is identical on objects, we know that $H$ is dense. We show that $H$ is faithful.

Let $\alpha:B\to C$ be any morphism in $\B/(\thick \W)$. It has the form $B\xrightarrow{\beta_0} D_1\xrightarrow{\beta_1} \cdot\cdot\cdot \xrightarrow{\beta_{n-1}} D_n\xrightarrow{\beta_n} C$ where $\beta_i$ is a morphism $\underline {f_i}$ or a morphism $\underline {g_i}^{-1}$ with $g_i\in \R$. We have a commutative diagram
$$\xymatrix@C=0.8cm@R0.7cm{
B \ar[d]_{\underline z_B^{-1}\underline b^+} \ar[r]^{\beta_0} &D_1 \ar[r]^{\beta_1} \ar[d] &\cdot \cdot \cdot \ar@{}[d]^{\cdots}_{\cdots} \ar[r]^{\beta_{n-1}} &D_n\ar[r]^{\beta_n} \ar[d] &C \ar[d]^{\underline z_C^{-1}\underline c^+}\\
Z_B \ar[r]^{\underline z_0} &Z_1 \ar[r]^{\underline z_1} &\cdot \cdot \cdot \ar[r]^{\underline z_{n-1}} &Z_n\ar[r]^{\underline z_n} &Z_C
}
$$
where $Z_i\in \Z$ and $z_i$ are morphisms in $\Z$ by Lemma \ref{lem}. Denote $z_0 z_1\cdot\cdot\cdot z_{n-1}z_n$ by $\zeta$, we have $\alpha=(\underline c^+)^{-1}\underline z_C\underline \zeta\underline z_B^{-1}\underline b^+$. If there exists a morphism $\alpha':B\to C$ in $\B/(\thick \W)$ such that $H(\alpha)=H(\alpha')$, since we also have $\alpha'=(\underline c^+)^{-1}\underline z_C\underline \zeta'\underline z_B^{-1}\underline b^+$ with some $\zeta'\in \Hom_\B(Z_B,Z_C)$, we can get that $\overline {\zeta}=H(\underline {\zeta})=H(\underline {\zeta'})=\overline {\zeta'}$. Hence $\zeta-\zeta'$ factors through $\W$, which implies $\underline {\zeta}=\underline {\zeta'}$. Thus $\alpha=\alpha'$.

Finally we show that $H$ is full. Let $\gamma:H(B)\to H(C)$ be any morphism. By Lemma \ref{lem} we can get that $\gamma=\overline z$ where $z$ is a morphism in $\Z$. Since we have the following commutative diagram in $\B/(\thick \W)$:
$$\xymatrix@C=0.8cm@R0.7cm{
B \ar[r]^{\underline b^+} \ar[d]_{\alpha} &B^+ \ar[d]^{\underline {z_Czz_B^{-1}}} &Z_B=H(B) \ar[l]_-{\underline z_B} \ar[d]^{\underline z}\\
C \ar[r]_{\underline c^+} &C^+ &Z_C=H(C) \ar[l]^-{\underline z_C}
}
$$
we have $H(\alpha)=H(\underline c^+)^{-1}H(\underline z_C)H(\underline z)H(\underline z_B)^{-1}H(\underline b^+)$. Since $H(\underline c^+)=H(\underline z_C)=\overline 1_{Z_C}$ and $H(\underline b^+)=H(\underline z_B)=\overline 1_{Z_B}$ by Remark \ref{rem1}, we obtain that $H(\alpha)=H(\underline z)=\overline z$, hence $H$ is full.
\end{proof}

Since we also have the following commutative diagram
$$\xymatrix{
\Z\ar@{ >->}[r]^{\eta} \ar@{->>}[dr]_{\pi} &\B \ar@{->>}[r]^-{L_\R} &\B/(\thick \W)\\
 &\Z/[\W] \ar@{.>}[ur]_{F}
}
$$
where $\eta$ is the embedding functor and $\pi$ is the canonical quotient functor, we have the following proposition.

\begin{prop}\label{quasi}
$F$ is the quasi-inverse of $H$.
\end{prop}

\begin{proof}
By definition we can get $HF=\Id_{\Z/[\W]}$. Let $f:B\to C$ be any morphism in $\B$. Then $FH(B)\xrightarrow{FH(\underline f)}FH(C)$ is just $Z_B\xrightarrow{\underline z_f}Z_C$. From the following commutative diagram
$$\xymatrix{
FH(B) \ar[rr]^{(\underline b^+)^{-1} \underline z_B}_{\simeq} \ar[d]_{FH(\underline f)} &&B \ar[d]^{\underline f}\\
FH(C) \ar[rr]_{(\underline c^+)^{-1} \underline z_C}^{\simeq} &&C
}
$$
we get that $FH\cong \Id_{\B/(\thick\W)}$.
\end{proof}

We give an example of our silting reduction.

\begin{exm}
Let $Q$ be the following infinite quiver:
$$\xymatrix@C=0.5cm@R0.5cm{\cdots \ar[r]^{x_{-5}} &-4 \ar[r]^{x_{-4}} &-3 \ar[r]^{x_{-3}} &-2 \ar[r]^{x_{-2}} &-1 \ar[r]^{x_{-1}} &0 \ar[r]^{x_{0}} &1 \ar[r]^{x_1} &2 \ar[r]^{x_2} &3 \ar[r]^{x_3} &4 \ar[r]^{x_4} &\cdots}$$
Let $\Lambda=kQ/(x_ix_{i+1}x_{i+2},i\neq 4k,x_{4k}x_{4k+1})$. Then the AR-quiver of $\mod \Lambda$ is the following.
{\small 
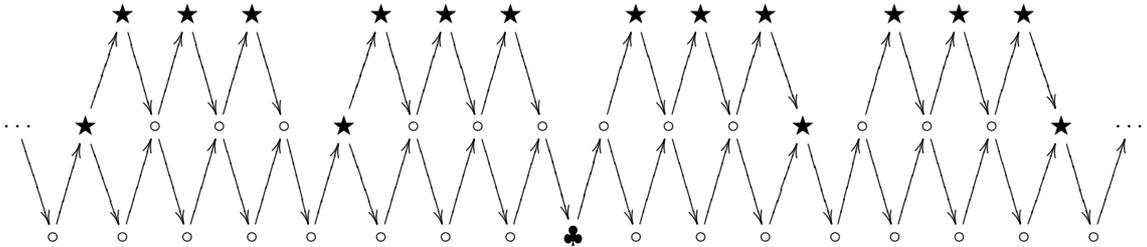
$$\xymatrix@C=0.001cm@R1cm{
 &&&\bigstar \ar[dr] &&\bigstar \ar[dr] &&\bigstar \ar[dr] &&&&\bigstar \ar[dr] &&\bigstar \ar[dr] &&\bigstar \ar[dr] &&&&\bigstar \ar[dr] &&\bigstar \ar[dr] &&\bigstar \ar[dr] &&&&\bigstar \ar[dr] && \bigstar \ar[dr] && \bigstar \ar[dr]\\
\cdots \ar[dr] &&\bigstar \ar[dr] \ar[ur] &&\circ \ar[ur] \ar[dr] &&\circ  \ar[ur] \ar[dr] &&\circ \ar[dr] &&\bigstar \ar[ur] \ar[dr] &&\circ \ar[ur] \ar[dr] &&\circ \ar[ur] \ar[dr] &&\circ  \ar[dr] &&\circ \ar[ur] \ar[dr] &&\circ \ar[ur] \ar[dr] &&\circ \ar[ur] \ar[dr] &&\bigstar \ar[dr] &&\circ \ar[ur] \ar[dr] &&\circ \ar[dr] \ar[ur] &&\circ \ar[dr] \ar[ur] &&\bigstar \ar[dr] &&\cdots \\
&\circ \ar[ur] &&\circ \ar[ur] &&\circ\ar[ur] &&\circ \ar[ur]  &&\circ \ar[ur] &&\circ\ar[ur] &&\circ \ar[ur] &&\circ \ar[ur] &&\clubsuit \ar[ur] &&\circ \ar[ur] &&\circ\ar[ur] &&\circ \ar[ur]  &&\circ \ar[ur] &&\circ \ar[ur] &&\circ \ar[ur] &&\circ \ar[ur] &&\circ \ar[ur]
}
$$}The additive closure of the indecomposable objects denoted by $\bigstar$ and the only object denoted by $\clubsuit$ form a silting subcategory in $\mod \Lambda$, denote it by $\s$. We denote the additive closure of the indecomposable objects in $\bigstar$ by $\W$, it is presilting and $(\W^{\vee},\W^{\bot})$, $({^{\bot}}\W,\W^{\wedge})$ are cotorsion pairs. The indecomposable objects in $\Z/[\W]$ are denoted by $\spadesuit$ in the diagram:
{\small 
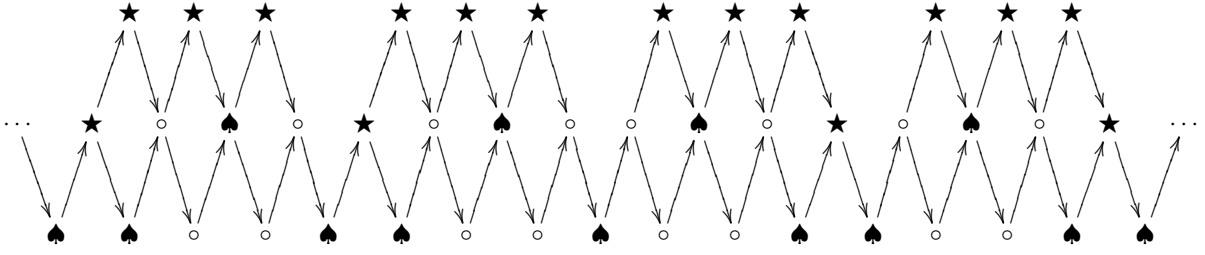
$$\xymatrix@C=0.001cm@R1cm{
 &&&\bigstar \ar[dr] &&\bigstar \ar[dr] &&\bigstar \ar[dr] &&&&\bigstar \ar[dr] &&\bigstar \ar[dr] &&\bigstar \ar[dr] &&&&\bigstar \ar[dr] &&\bigstar \ar[dr] &&\bigstar \ar[dr] &&&&\bigstar \ar[dr] && \bigstar \ar[dr] && \bigstar \ar[dr]\\
\cdots \ar[dr] &&\bigstar \ar[dr] \ar[ur] &&\circ \ar[ur] \ar[dr] &&\spadesuit \ar[ur] \ar[dr] &&\circ \ar[dr] &&\bigstar \ar[ur] \ar[dr] &&\circ \ar[ur] \ar[dr] &&\spadesuit \ar[ur] \ar[dr] &&\circ  \ar[dr] &&\circ \ar[ur] \ar[dr] &&\spadesuit \ar[ur] \ar[dr] &&\circ \ar[ur] \ar[dr] &&\bigstar \ar[dr] &&\circ \ar[ur] \ar[dr] &&\spadesuit\ar[dr] \ar[ur] &&\circ \ar[dr] \ar[ur] &&\bigstar \ar[dr] &&\cdots \\
&\spadesuit \ar[ur] &&\spadesuit \ar[ur] &&\circ\ar[ur] &&\circ \ar[ur]  &&\spadesuit \ar[ur] &&\spadesuit \ar[ur] &&\circ \ar[ur] &&\circ \ar[ur] &&\spadesuit \ar[ur] &&\circ \ar[ur] &&\circ\ar[ur] &&\spadesuit \ar[ur]  &&\spadesuit \ar[ur] &&\circ \ar[ur] &&\circ \ar[ur] &&\spadesuit \ar[ur] &&\spadesuit \ar[ur]
}
$$}
$\s$ is a silting subcategory in $\Z$ and the object denoted by $\clubsuit$ is a silting object in $\Z/[\W]$. Moreover, any subcategory consisted by the direct sums of objects denoted by $\bigstar$ and one fixed object  denoted by $\spadesuit$ is a silting subcategory in both $\Z$ and $\B$.
\end{exm}

\section{The triangle equivalence}

In this section, we show that $\B/(\thick \W)$ has a natural triangulated structure.

\begin{prop}\label{induce1}
Let $A\xrightarrow{x} B\xrightarrow{y} C\dashrightarrow$ be any $\EE$-triangle in $\B$. There is a commutative diagram
$$\xymatrix@C=0.8cm@R0.7cm{
A\ar[r]^{\underline x} \ar[d]^{\simeq} &B\ar[r]^{\underline y} \ar[d]^{\simeq} &C\ar[d]^{\simeq}\\
\widetilde{Z_A} \ar[r]_{\underline {\widetilde{z_x}}} &\widetilde{Z_B} \ar[r]_{\underline {\widetilde{z_y}}} &\widetilde{Z_C}
}
$$
in $\B/(\thick \W)$, where $\widetilde{Z_A}\xrightarrow{\underline {\widetilde{z_x}}} \widetilde{Z_B}\xrightarrow{\underline {\widetilde{z_y}}} \widetilde{Z_C}$ admits the following commutative diagram in $\Z$:
$$\xymatrix@C=0.8cm@R0.7cm{
\widetilde{Z_A} \ar@{=}[d] \ar[r]^{\widetilde{z_x}} &\widetilde{Z_B} \ar[d] \ar[r]^{\widetilde{z_y}} &\widetilde{Z_C} \ar[d] \ar@{-->}[r]&\\
\widetilde{Z_A} \ar[r] &\widetilde{W} \ar[r] &Z_A\langle 1\rangle \ar@{-->}[r]&
}
$$
where $\widetilde{W}\in \W$ and $Z_A=G(A)$.
\end{prop}

\begin{proof}
We have the following commutative diagrams
$$\xymatrix@C=0.8cm@R0.7cm{
A \ar[r]^x \ar[d]_{a^+} &B \ar[r]^y \ar[d]^{\widetilde{b}} &C\ar@{-->}[r] \ar@{=}[d] &\\
A^+ \ar[r]_{\widetilde{x}} \ar[d] &\widetilde{B} \ar[d] \ar[r]_{\widetilde{y}} &C \ar@{-->}[r] &\\
U_A \ar@{=}[r] \ar@{-->}[d] &U_A \ar@{-->}[d]\\
&&\\
}\quad \quad \xymatrix@C=0.8cm@R0.7cm{
A^+ \ar[r]^{\widetilde{x}} \ar@{=}[d] &\widetilde{B} \ar[r]^{\widetilde{y}} \ar[d]^{\widetilde{b}^+} &C \ar[d]^{(c^+)'} \ar@{-->}[r] &\\
A^+ \ar[r] &\widetilde{B}^+ \ar[r] \ar[d] &(C^+)' \ar[d] \ar@{-->}[r] &\\
&U_{\widetilde{B}} \ar@{-->}[d] \ar@{=}[r] &U_{\widetilde{B}} \ar@{-->}[d]\\
&&&
}
$$
$$\xymatrix@C=0.8cm@R0.7cm{
&(V^C)' \ar[d] \ar@{=}[r] &(V^C)' \ar[d]\\
A^+ \ar[r]^{\widehat{x}} \ar@{=}[d] &(\widetilde{B}^+)' \ar[r]^{\widehat{y}} \ar[d]^{\widehat{b}} &\widetilde{Z_C} \ar[d]^{\widetilde{z_C}} \ar@{-->}[r] &\\
A^+ \ar[r] &\widetilde{B}^+ \ar[r] \ar@{-->}[d] &(C^+)' \ar@{-->}[d] \ar@{-->}[r] &\\
&&&\\
}\quad \quad \xymatrix@C=0.8cm@R0.7cm{
(V^B)' \ar[d] \ar@{=}[r] & (V^B)' \ar[d]\\
\widetilde{Z_A} \ar[r]^{\widetilde{z_x}} \ar[d]_{\widetilde{z_A}} &\widetilde{Z_B} \ar[r]^{\widetilde{z_y}} \ar[d]_{\widetilde{z_B}} &\widetilde{Z_C} \ar@{=}[d] \ar@{-->}[r] &\\
A^+ \ar[r] \ar@{-->}[d] &(\widetilde{B}^+)' \ar[r] \ar@{-->}[d] &\widetilde{Z_C} \ar@{-->}[r] &\\
&&
}
$$
where $U_A,U_{\widetilde{B}}\in \W^{\vee}$, $(V^C)',(V^B)'\in \W^{\wedge}$, $a^+,\widetilde{b}^+$ are minimal left $(\W^{\bot})$-approximations and $\widetilde{z_C},\widetilde{z_B}$ are minimal right $({^{\bot}}\W)$-approximations. Then we have the following commutative diagram.
$$\xymatrix@C=0.8cm@R0.7cm{
A\ar[r]^{\underline x} \ar[d]_{\underline{(\widetilde{z_A})^{-1}a^+}}^{\simeq} &B\ar[rrr]^{\underline y} \ar[d]^{ \underline{(\widetilde{z_B})^{-1}(\widehat{b})^{-1}\widetilde{b}^+\widetilde{b}}}_{\simeq} &&&C\ar[d]^{\underline {(\widetilde{z_C})^{-1}(c^+)'}}_{\simeq}\\
\widetilde{Z_A} \ar[r]_{\underline {\widetilde{z_x}}} &\widetilde{Z_B} \ar[rrr]_{\underline {\widetilde{z_y}}} &&&\widetilde{Z_C}
}
$$
We also have the following commutative diagram
$$\xymatrix@C=0.8cm@R0.7cm{
V^A \ar[r] \ar[d] &Z_A \ar[r]^{z_A} \ar[d]^r &A^+ \ar@{=}[d] \ar@{-->}[r] &\\
(V^B)' \ar[r] \ar[d] &\widetilde{Z_A} \ar[r]^{\widetilde{z_A}} \ar[d]^s &A^+ \ar@{-->}[r] \ar@{=}[d] &\\
V^A \ar[r] &Z_A \ar[r]^{z_A} &A^+ \ar@{-->}[r]&
}
$$
where $z_A$ is a minimal right $({^{\bot}}\W)$-approximation. Hence $sr$ is an isomorphism and $Z_A$ is a direct summand of $\widetilde{Z_A}$. For convenience, we can assume that $sr=1$. On the other hand, $\widetilde{Z_A}$ is a direct summand of $Z_A\oplus V^A$ since $1-rs$ factors through $V_A$. Thus $\widetilde{Z_A}\simeq Z_A\oplus W^A$ with $W^A\in \W$. Let $\widetilde{Z_A}\xrightarrow{\svecv{s}{u}}Z_A\oplus W^A$ be an isomorphism and $Z_A\oplus W^A\xrightarrow{\svech{r}{v}}\widetilde{Z_A}$ be its inverse. We have the following commutative diagram.
$$\xymatrix@C=0.8cm@R0.7cm{
\widetilde{Z_A} \ar[r]^{\widetilde{z_x}} \ar[d]_{\svecv{s}{u}} &\widetilde{Z_B} \ar[r]^{\widetilde{z_y}} \ar[d] &\widetilde{Z_C} \ar[d]^{a} \ar@{-->}[r] &\\
Z_A\oplus W^A \ar[d]_{\svech{r}{v}} \ar[r]^-{{\left(\begin{smallmatrix}
w_A&0\\
0&1\\
\end{smallmatrix}\right)}} &W_A\oplus W^A \ar[d]^-{{\left(\begin{smallmatrix}
1&0\\
0&1\\
\end{smallmatrix}\right)}} \ar[r]^-{\svech{v_A}{0}} &Z_A\langle1\rangle \ar@{-->}[r] \ar[d]^b &\\
\widetilde{Z_A} \ar[r]_-{\svecv{w_As}{u}} &W_A\oplus W^A \ar[r]_-{\svech{v_A}{0}} &Z_A\langle1\rangle \ar@{-->}[r]&
}
$$
For convenience, we denote morphism $ba$ by $\widetilde{z}$.
\end{proof}

\begin{rem}
Denote $F\circ\langle1\rangle\circ H$ by $\underline {[1]}$. Note that for any object $A$, $A\underline{[1]}=Z_A\langle1\rangle$ in $\B/(\thick\W)$.
We have the following commutative diagram in $\B/(\thick\W)$
$$\xymatrix@C=0.8cm@R0.7cm{
A\ar[r]^{\underline x} \ar[d]^{\simeq} &B\ar[r]^{\underline y} \ar[d]^{\simeq} &C\ar[d]^{\simeq} \ar[rr]^-{\underline {\widetilde{z}(\widetilde{z_C})^{-1}(c^+)'}} &&A\underline{[1]} \ar@{=}[d]\\
\widetilde{Z_A} \ar[r]_{\underline {\widetilde{z_x}}} &\widetilde{Z_B} \ar[r]_{\underline {\widetilde{z_y}}} &\widetilde{Z_C} \ar[rr]_{\underline {\widetilde{z}}} &&Z_A\langle1\rangle
}
$$
where $\widetilde{Z_A} \xrightarrow{\overline {\widetilde{z_x}}} \widetilde{Z_B} \xrightarrow{\overline {\widetilde{z_y}}} \widetilde{Z_C} \xrightarrow{\overline {\widetilde{z}}} Z_A\langle1\rangle$ is a distinguished triangle in $\Z/[\W]$. Moreover, $(\underline{(\widetilde{z_A})^{-1}a^+})\underline{[1]}=\underline 1_{A\underline{[1]}}$
\end{rem}

\begin{defn}\label{tri}
We call a sequence $X\xrightarrow{\alpha} Y\xrightarrow{\beta} Z\xrightarrow{\gamma} X\underline{[1]}$ a triangle in $\B/(\thick\W)$ if we have an commutative diagram in $\B/(\thick\W)$
$$\xymatrix@C=0.8cm@R0.7cm{
X \ar[r]^{\alpha} \ar[d]_{\simeq}^{\phi_1} &Y \ar[r]^{\beta} \ar[d]_{\simeq}^{\phi_2} &Z \ar[rr]^{\gamma} \ar[d]_{\simeq}^{\phi_3} &&X\underline{[1]} \ar[d]_{\simeq}^{\phi_1\underline{[1]}}\\
A \ar[r]_{\underline x} &B \ar[r]_{\underline y} &C \ar[rr]_{\underline {\widetilde{z}(\widetilde{z_C})^{-1}(c^+)'}} &&A\underline{[1]}
}
$$
where the second row is induced by an $\EE$-triangle $A\xrightarrow{x} B\xrightarrow{y} C\dashrightarrow$ as in Proposition \ref{induce1}. The triangles induced by $\EE$-triangles in $\B$ are called distinguished triangles in $\B/(\thick \W)$.
\end{defn}

\begin{thm}\label{main6}
$\B/(\thick \W)$ is a triangulated category with triangles defined in Definition \ref{tri} and shift functor $\underline{[1]}$. Moreover, the equivalence $H$ is a triangle equivalence.
\end{thm}

\begin{proof}
We know $\underline {[1]}=F\circ \langle 1 \rangle\circ H$ is an auto-equivalence. Moreover, by Proposition \ref{quasi}, we have $\underline{[1]}\circ F=F\circ  \langle 1 \rangle$ and $H\circ \underline {[1]} \cong \langle 1 \rangle\circ H$.

Since $HF=\Id_{\Z/[\W]}$, and $F$ is identical on objects, the image of $\Z/[\W]$ is a triangulated subcategory in $\B/(\thick\W)$ with shift functor $\underline {[1]}|_{\Z}$ and distinguished triangles which are induced by the images of triangles in $\Z/[\W]$. Since $\Z$ is dense in $\B/(\thick\W)$, by definition and Proposition \ref{induce1}, any triangle in $\B/(\thick\W)$ is isomorphic to the image of a triangle in $\Z/[\W]$. To show that $\B/(\thick \W)$ is triangulated, we only need to check that any morphism $X\xrightarrow{\alpha} Y$ in $\B/(\thick \W)$ admits a triangle $X\xrightarrow{\alpha} Y\xrightarrow{\beta} Z\xrightarrow{\gamma} X\underline{[1]}$.

By the proof of Theorem \ref{main5}, we have the following commutative diagram
$$\xymatrix@C=0.8cm@R0.7cm{
X\ar[r]^{\alpha} \ar[d]_{\underline {z_X^{-1}x^+}}^{\simeq}  &Y \ar[d]^{\underline {z_Y^{-1}y^+}}_{\simeq}\\
Z_X \ar[r]_{\underline z_1} &Z_Y
}
$$
where $z_1\in \Hom_{\B}(Z_X,Z_Y)$. $z_1$ admits the following commutative diagram
$$\xymatrix@C=0.8cm@R0.7cm{
Z_X \ar[d]_{z_1} \ar[r] &W_X \ar[r] \ar[d] &Z_X\langle1\rangle \ar@{=}[d] \ar@{-->}[r] &\\
Z_Y \ar[r]_{z_2} &Z \ar[r]_{z_3} &Z_X\langle1\rangle \ar@{-->}[r] &
}
$$
with $W_X\in \W$ and $Z\in \Z$. Then we have the following commutative diagram
$$\xymatrix@C=0.8cm@R0.7cm{
Z_X \ar[r]^-{\svecv{z_1}{*}} \ar@{=}[d] &Z_Y\oplus W_X \ar[r]^-{\svech{z_2}{*}} \ar[d] &Z \ar[d]^{-z_3} \ar@{-->}[r] &\\
Z_X \ar[r] &W_X \ar[r] &Z_X\langle1\rangle \ar@{-->}[r] &
}
$$
which induces a distinguished triangle $Z_X\xrightarrow{\underline z_1}Z_Y\xrightarrow{\underline z_2} Z\xrightarrow{-\underline z_3} Z_X\langle1\rangle$ in $\B/(\thick \W)$. Then we have a commutative diagram.
$$\xymatrix@C=0.8cm@R0.7cm{
X\ar[r]^{\alpha} \ar[d]_{\underline {z_X^{-1}x^+}}^{\simeq}  &Y \ar[d]^{\underline {z_Y^{-1}y^+}}_{\simeq} \ar[rr]^-{\underline {z_2z_Y^{-1}y^+}} &&Z \ar@{=}[d] \ar[r]^{-\underline z_3} &X\underline{[1]} \ar@{=}[d]\\
Z_X \ar[r]_{\underline z_1} &Z_Y \ar[rr]_{\underline z_2} &&Z \ar[r]_{-\underline z_3} &Z_X\langle1\rangle
}
$$
Note that $(\underline {z_X^{-1}x^+})\underline {[1]}=\underline 1_{X\underline{[1]}}$,  we find that the sequence $X\xrightarrow{\alpha} Y \xrightarrow{\underline {zz_X^{-1}x^+}} Z \xrightarrow{-\underline z_3} X\underline{[1]}$ is the triangle we need.
\end{proof}

By Lemma \ref{psilz}, Corollary \ref{silzw}, Theorem \ref{main5} and Theorem \ref{main6}, we have the following corollary.

\begin{cor}\label{cor:corr}
There is a one-to-one correspondence between the presilting subcategories in $\B/(\thick\W)$ and the presilting subcategories in $\B$ which contain $\W$. This correspondence also induces a one-to-one correspondence between the silting subcategories in $\B/(\thick\W)$ and the silting subcategories in $\B$ which contain $\W$.
\end{cor}

\section{Silting subcategories vs tilting subcategories}

In this section, we discuss the relation between silting categories and tilting subcategories. This gives us a kind of important examples of the results in Section 3 and Section 4.

\begin{lem}\label{lem:leq+1}
Let $A\to B\to C\dashrightarrow$ be an $\EE$-triangle. Then we have $$\id A\leq \max\{\id B,\id C+1\},\ \id B\leq\max\{\id A,\id C\},\id C\leq \max\{\id B,\id A- 1 \},\text{ and}$$
$$\pd A\leq\max\{\pd B,\pd C-1\},\ \pd B\leq\max\{\pd A,\pd C\},\ \pd C\leq \max\{\pd A+1,\pd B\}.$$
\end{lem}

\begin{proof}
This follows directly from the exact sequence in Lemma~\ref{lem:long}.
\end{proof}

We give the following definition of tilting subcategories in extriangulated categories, which is sightly different from \cite[Defintion 7]{ZhZ}. Note that by \cite[Remark 4]{ZhZ}, the tilting subcategory defined in \cite[Defintion 7]{ZhZ} also satisfies the following definition.

\begin{defn}\label{deftil}
A subcategory $\T$ of $\B$ is called partial $n$-tilting ($n\geq 1$) if the following hold.
\begin{itemize}
\item[(P1)] $\pd \T\leq n$.
\item[(P2)] $\EE^i(\T,\T)=0,i=1,2,...,n$.
\end{itemize}
A partial $n$-tilting subcategory $\T$ is called \emph{$n$-tilting} if the following holds.
\begin{itemize}
\item[(P3)] $\mathcal P\subseteq\T_n^{\vee}$.
\end{itemize}

Any partial $n$-tilting (resp. $n$-tilting) subcategory is simply called a \emph{partial tilting} (resp. \emph{tilting}) subcategory. An object $T$ is called  an $n$-tilting object if $\add T$ is  an $n$-tilting subcategory. The notion of \emph{$n$-cotilting} subcategories can be defined dually.
\end{defn}

\begin{rem}
If $\B$ is the category of finitely generated modules of a finite-dimensional algebra, then the $n$-tilting objects in $\B$ are exactly the $n$-tilting modules. But if $\B$ is a (non-zero) triangulated category, then there is no non-zero tilting subcategory of $\B$ because in this case $\mathcal P=\{0\}$ and for any non-zero subcategory $\T$ of $\B$, we have $\pd\T=\infty$.
\end{rem}

We have the following observation.

\begin{lem}\label{lem:easy}
Any partial tilting subcategory of $\B$ is presilting. If $\mathcal I\subseteq\thick\mathcal P$, then any tilting subcategory of $\B$ is silting in $\thick\mathcal P$.
\end{lem}

\begin{proof}
Conditions (P1) and (P2) imply that $\EE^i(\T,\T)=0$ for any $i>0$. Then the first assertion holds. Conditions (P1) and (P3) imply that $\thick\mathcal P=\thick\T$. So by Lemma~\ref{main2cor}, the second assertion holds.
\end{proof}

\begin{prop}\label{main3}
Let $\T$ be a  partial $n$-tilting subcategory in $\B$. Consider the following conditions:
\begin{itemize}
\item[(a)] $\T$ is contravariantly finite and $n$-tilting;
\item[(b)] $(\T^{\vee},\T^{\bot})$ is a cotorsion pair;
\item[(c)] $\T$ is an $n$-tilting subcategory.
\end{itemize}
We have {\rm  (a)$\Rightarrow$(b)$\Rightarrow$(c).}
\end{prop}

\begin{proof}

(a)$\Rightarrow$(b): Let $\T$ be a contravariantly finite $n$-tilting subcategory. By \cite[Proposition 4.3]{MDZH}, we only need to show that:
\begin{itemize}
\item[(1)] $\B=(\T^{\bot})^{\vee}$.
\item[(2)] Any object $V\in \T^{\bot}$ admits an $\EE$-triangle $V'\to T\to V\dashrightarrow$ with $T\in \T,V'\in \T^{\bot}$.
\end{itemize}

(1) For any $B\in \B$, since $\B$ has enough injectives, we have the following $\EE$-triangles
$$B_{i-1}\to I_i\to B_i\dashrightarrow,\ i=1,...,n,$$
with $I_i\in \mathcal I$ and $B_0=B$. By Lemma~\ref{lem:long}, there are exact sequences
$$0=\EE^j(\T,I_i)\to\EE^j(\T,B_i)\to\EE^{j+1}(\T,B_{i-1})\to\EE^{j+1}(\T,I_i)=0,\ 1\leq i\leq n,j\geq 1.$$
Hence we have isomorphisms
$$\EE^j(\T,B_i)\cong\EE^{j+1}(\T,B_{i-1}),\ 1\leq i\leq n,j\geq 1.$$
Since $\pd \T\leq n$, we have $\EE^{j+1}(\T,B_0)=0$ for any $j\geq n$. So by induction, we have $\EE^{j}(\T,B_n)=0$ for any $j\geq 1$. This implies that $B_n\in\T^\perp$. Since $\mathcal I\subseteq \T^{\bot}$, we have $B\in (\T^{\bot})^{\vee}$. Hence $\B=(\T^{\bot})^{\vee}$.

(2) Any object $V\in \T^{\bot}$ admits a deflation $P\xrightarrow{p} V$ with $P\in \mathcal P$. Since $\T$ is an $n$-tilting subcategory, we have $\EE$-triangles
$$R_i\xrightarrow{r_i} T_i\to R_{i+1}\dashrightarrow,\ ,i=0,1,..., n-1$$
where $R_0=P,T_i\in \T, R_n\in \T$. By Lemma~\ref{lem:long}, there are exact sequences
$$0=\EE^{j}(T_i,\T^{\bot})\to \EE^{j}(R_i,\T^{\bot})\to \EE^{j+1}(R_{i+1},\T^{\bot})\to \EE^{j+1}(T_i,\T^{\bot})=0,\  0\leq i< n,j\geq 1.$$
Hence we have isomorphisms
$$\EE^{j}(R_i,\T^{\bot})\cong \EE^{j+1}(R_{i+1},\T^{\bot}),\  0\leq i<n,j\geq 1.$$
Since $\EE^n(R_n,\T^\perp)=0$, by induction, we have $\EE(R_1,\T^\perp)=0$. It follows that $r_0$ is a left $(\T^{\bot})$-approximation of $R_0=P$. Hence there exists a morphism $t_0:T_0\to V$ such that $p=t_0r_0$. Using condition (WIC), we get that $t_0$ is a deflation.
Since $\T$ is contravariantly finite, $V$ admits a right $\T$-approximation $t:T\to V$, then $t_0$ factors through $t$ with $T\in\T$. Again by condition (WIC), we get that $t$ is a deflation. So there is an $\EE$-triangle
$$V'\to T\xrightarrow{t} V\dashrightarrow.$$
Since $t$ is a right $\T$-approximation and $\T$ is  partial tilting, we have that $V'\in \T^{\bot_1}$. By Lemma~\ref{lem:long}, there are exact sequences
$$0=\EE^{i}(\T,T)\to \EE^{i}(\T,V)\to \EE^{i+1}(\T,V')\to \EE^{i+1}(\T,T)=0,\  i\geq 1,$$
which implies that $\EE^i(\T,V')=0$, $i\geq 2$. Hence $V'\in \T^{\bot}$.

(b)$\Rightarrow$(c): If $(\T^{\vee},\T^{\bot})$ is a cotorsion pair, then by Lemma~\ref{lem:basic}~(4), we have $\mathcal P\subseteq \T^{\vee}$. So any object $P\in \mathcal P$ admits $\EE$-triangles
$$R_i\to T_i\to R_{i+1}\dashrightarrow,\ i=0,1,..., n-1,$$
with $R_0=P,T_i\in \T$. By Lemma~\ref{lem:long}, there are exact sequences
$$0=\EE^i(\T^{\vee},T_{n-i})\to \EE^i(\T^{\vee},R_{n-i+1})\to \EE^{i+1}(\T^{\vee},R_{n-i})\to \EE^{i+1}(\T^{\vee},T_{n-i})=0,\  i=1,2,...,n$$
which implies $\EE^i(\T^{\vee},R_{n-i+1})\cong \EE^{i+1}(\T^{\vee},R_{n-i}), i=1,2,..., n.$ Since $\EE^{n+1}(\T^{\vee},R_0)=0$ by (P1), we get that $\EE(\T^{\vee},R_{n})=0$. So we have $R_{n}\in \T^{\perp}$ by Lemma~\ref{lem:basic}~(1). Since by definition, $R_{n}$ also lies in $\T^{\vee}$, we can get $R_{n}\in \T^\vee\cap\T^\perp=\T$. It follows that $P\in\T_n^{\vee}$. Hence (P3) holds.
\end{proof}

Let $\sil\B$ (resp. $\til\B$) be the class of all the silting (resp. tilting) subcategories in $\B$.

\begin{thm}\label{main4.4}
The following statements hold.
\begin{itemize}
\item[(1)] If for any object $B\in \B$ we have $\pd B<\infty$, then $\sil\B\supseteq \til\B$.
\item[(2)] If $\pd \B<\infty$, then $\sil\B= \til\B$.
\item[(3)] If there exists an object $X$ of $\B$ such that $\pd X=\infty$, then $\sil\B\cap \til\B=\emptyset$.
\end{itemize}
\end{thm}

\begin{proof}
(1) If any object $B\in \B$ has finite projective dimension, then $\thick \mathcal P=\B$. By Lemma~\ref{lem:easy}, we get the assertion.

(2) If $\pd \B=n<\infty$, by (1), we only need to show $\sil\B\subseteq \til\B$. Let $\s\in\sil\B$. Then $\s$ is a partial tilting subcategory. By Theorem~\ref{main2}, $(\s^{\vee},\s^{\wedge})$ is a cotorsion pair. By the proof of ``(b)$\Rightarrow$(c)'' in Proposition \ref{main3}, we can get that $\mathcal P\subseteq \s^{\vee}_n$. Hence by definition, $\s$ is a tilting subcategory.

(3) If there is an object having infinity projective dimension, then $\thick\mathcal P\subsetneq\B$. This means that for any subcategory $\s$ of $\B$,  $\thick\s=\thick\mathcal P$ and $\thick\s=\B$ can not hold at the same time. It follows that there is no subcategory which is both silting and tiling.
\end{proof}

We are interested in the case when $\B$ contains a tilting-cotilting subcategory, which gives an important example of the results in Section 4. We generalize \cite[Definition 7.2.1]{BR} to extriangulated categories.

\begin{defn}\label{Gorenstein}
$\B$ is called \emph{Gorenstein} if $\pd\mathcal I<\infty$ and $\id\mathcal P<\infty$.
\end{defn}

By definition, we have the following observation.

\begin{rem}\label{lem:Gor}
If $\B$ is Gorenstein, then $\mathcal P$ and $\mathcal I$ are tilting-cotilting subcategories of $\B$.
\end{rem}

Consider the following condition.
\begin{description}
\item[(AF)] Either $\B$ is an abelian category, or $\mathcal P$ is covariantly finite and $\mathcal I$ is contravariantly finite.
\end{description}

We have the following lemma.

\begin{lem}\label{lem:eq}
Let $\T$ be a  partial tilting subcategory in $\B$. Under the condition $\operatorname{(AF)}$, if $\B$ is Gorenstein, then the following are equivalent:
\begin{itemize}
\item[(1)] $\T$ is a tilting subcategory;
\item[(2)] $(\T^{\vee},\T^{\bot})$ is a cotorsion pair;
\item[(3)] $\T$ is a cotilting subcategory;
\item[(4)] $({^{\bot}}\T,\T^{\wedge})$ is a cotorsion pair.
\end{itemize}
\end{lem}

\begin{proof}
(1)$\Rightarrow$(4): By Lemma~\ref{lem:easy}, we have $\thick \T=\thick \mathcal P$. Since $\B$ is Gorenstein, we have $\thick\mathcal P=\thick\mathcal I$ which has enough projectives and enough injectives. Then by Lemma~\ref{main2cor}, $\mathcal P$, $\mathcal I$ and $\T$ are silting subcategories in $\thick\mathcal P$. So by Theorem~\ref{main2}, $(\T^{\vee},\T^{\wedge})$ is a cotorsion pair in $\thick \mathcal P$. We claim that $({^{\bot}\mathcal P},\thick \mathcal P)$ is a cotorsion pair in $\B$. Indeed, if $\B$ is an abelian category, this is due to \cite[Theorem 7.2.2]{BR}. If $\mathcal P$ is covariantly finite, since by Remark~\ref{lem:Gor}, $\mathcal P$ is cotilting in $\B$, by the dual of Proposition \ref{main3}, we have that $({^{\bot}\mathcal P}, \mathcal P^{\wedge}=\thick\mathcal P)$ is a cotorsion pair. Thus, we finish the proof of the claim. Then $\T^{\wedge}$ is covariantly finite in $\B$. Hence by the dual of \cite[Proposition 3.4]{CZZ}, $({^{\bot}}\T,\T^{\wedge})$ is a cotorsion pair in $\B$.
	
(4)$\Rightarrow$(3): By Lemma~\ref{lem:leq+1}, we have $\id\T\leq\id\mathcal P<\infty$. Then (3) follows from the dual of Proposition~\ref{main3}.
	
(3)$\Rightarrow$(2): this is the dual of the proof of ``(1)$\Rightarrow$(4)".
	
(2)$\Rightarrow$(1): it follows from Proposition \ref{main3}.
\end{proof}

\begin{rem}
	If one of $\mathcal P$ and $\mathcal I$ contains finitely many indecomposable non-isomorphic objects, then by \cite[Proposition~5.8]{AT}, both of them contains finitely many indecomposable non-isomorphic objects. So in this case, assumption $\operatorname{(AF)}$ holds.
	\end{rem}

The following criterion on $\B$ having tilting-cotilting subcategories is a generalization of
a result of Happel-Unger \cite[Lemma 1.3]{HU} and a result of \cite[Theorem 2.6]{XZZ}.

\begin{thm}\label{main4}
Under assumption~$\operatorname{(AF)}$, the following are equivalent:
\begin{itemize}
\item[(1)] $\B$ is Gorenstein;
\item[(2)] any tilting subcategory is cotilting;
\item[(3)] any cotilting subcategory is tilting;
\item[(4)] $\B$ has a tilting-cotilting subcategory.
\end{itemize}
\end{thm}

\begin{proof}
``(2)$\Rightarrow$(4)": Since $\mathcal P$ is a tilting subcategory of $\B$, by (2) we have that $\mathcal P$ is cotilting. So we have (4).

``(4)$\Rightarrow$(1)": By (4), there is an $n$-tilting, $m$-cotilting subcategory $\T$ in $\B$. By definition every projective object $P\in \mathcal P$ admits $\EE$-triangles
$$R_i\to T_i\to R_{i+1}\dashrightarrow,\ i=0,1,\cdots,n-1$$
with $R_0=P,T_i\in \T,R_{n}\in \T$. Since $\id \T\leq m$, by Lemma~\ref{lem:leq+1}, we have $\id P\leq m+n$. Hence $\id\mathcal P\leq m+n$. Dually, we have $\pd\mathcal I\leq m+n$, too. Thus, we get (1).

``(1)$\Rightarrow$(2)": this follows  directly from Lemma~\ref{lem:eq}.

Dually, one can show the implications (1)$\Rightarrow$(3)$\Rightarrow$(4)$\Rightarrow$(1).
\end{proof}

By Theorem \ref{main5}, Theorem \ref{main6}, Proposition \ref{main3} and its dual, we have the following corollary:

\begin{cor}
Let $\T$ be a tilting subcategory of $\B$. If  $\B$ is Gorenstein and satisfies condition {\rm (AF)}, then we have the following triangle equivalences:
$${^{\bot}}\mathcal P/[\mathcal P]\simeq \B/(\thick\mathcal P)=\B/(\thick \T)=\B/(\thick\mathcal I)\simeq \mathcal I^{\bot}/[\mathcal I]\simeq (\T^{\bot}\cap {^{\bot}}\T)/[\T].$$
\end{cor}

Let $A$ be a Gorenstein algebra and $T$ be an $n$-tilting $A$-module. According to \cite[Lemma 1.5]{H}, we have $\mathrm{D}^b(A)/\mathrm{K}^b(\add T)=\mathrm{D}^b(A)/\mathrm{K}^b(\proj A)$. By \cite[Theorem 4.6]{H}, there is a triangle equivalence $$\mathrm{D}^b(A)/\mathrm{K}^b(\proj A)\simeq {^{\bot}}A/[A].$$
Now applying the corollary above, we can get the following corollary, where the second equivalence appeared in \cite[Theorem 2.5]{CZ}.

\begin{cor}
Let $A$ be a Gorenstein algebra and $T$ be an $n$-tilting $A$-module. We have the following triangle equivalences:
$$\mod A/(\thick T)\simeq (T^{\bot}\cap {^{\bot}}T)/[T] \simeq \mathrm{D}^b(A)/\mathrm{K}^b(\add T).$$
\end{cor}

\end{document}